\documentclass[preprint,12pt]{article}
\usepackage{amsfonts}
\usepackage{bbm}
\usepackage[top=1in, bottom=1in, left=1in, right=1in]{geometry}
\usepackage{amsmath,booktabs,ctable,threeparttable}
\usepackage{amssymb,amsfonts,boxedminipage}
\usepackage{amsthm}
\usepackage{algorithm}
\usepackage{algorithmic}
\usepackage{epsfig,graphicx,picins,picinpar,subfigure, epstopdf}
\usepackage{multirow}
\usepackage{verbatim}
\usepackage{bm}
\usepackage[round,numbers,authoryear]{natbib}
\usepackage[colorlinks,
            linkcolor=red,
            anchorcolor=blue,
            citecolor=blue
            ]{hyperref}

\numberwithin{equation}{section}
\newtheorem{theorem}{Theorem}[section]

\newtheorem{lemma}[theorem]{Lemma}
\newtheorem{proposition}[theorem]{Proposition}
\theoremstyle{definition}
\newtheorem{assumption}[theorem]{Assumption}
\newtheorem{definition}[theorem]{Definition}

\newtheorem{remark}[theorem]{Remark}

\makeatletter
 
 \newcommand{\abs}[1]{\left\vert#1\right\vert}

\newcommand{\E}[1]{\mathbb{E}[#1]}

\newcommand{\var}[1]{\mathrm{Var}(#1)}

\newcommand{\Unif}[1]{\mathbb{U}#1}

\newcommand{\ceil}[1]{\lceil #1 \rceil}
\newcommand{\mrd}{\mathrm{d}}
\newcommand{\mbe}{\mathbb{E}}
\newcommand{\mbp}{\mathbb{P}}
\newcommand{\nat}{\mathbb{N}}
\def\BV{\mathrm{HK}}
\makeatother

\begin{document}
\title{Sensitivity estimation of conditional value at risk using randomized quasi-Monte Carlo}
\author{Zhijian He\\South China University of Technology}
\maketitle
\begin{abstract}
Conditional value at risk (CVaR) is a popular measure for quantifying portfolio risk. Sensitivity analysis of CVaR is very useful in risk management and gradient-based optimization algorithms.
In this paper, we study the infinitesimal perturbation analysis estimator for CVaR sensitivity using randomized quasi-Monte Carlo (RQMC) simulation. We first prove that the RQMC-based estimator is strongly consistent under very mild conditions.  Under some technical conditions, RQMC that uses $d$-dimensional points in CVaR sensitivity estimation yields a mean error rate of $O(n^{-1/2-1/(4d-2)+\epsilon})$ for arbitrarily small $\epsilon>0$. The numerical results show that the RQMC method performs better than the Monte Carlo method for all cases. The gain of plain RQMC deteriorates  as the dimension $d$ increases, as predicted by the established theoretical error rate.

\end{abstract}

\section{Introduction}\label{se:intro}
In the financial industry, value at risk (VaR) and conditional VaR (CVaR) are two important tools for quantifying and managing portfolio risk. From the view of statistics, VaR is  a quantile of a portfolio's loss (or profit) over a holding period. On the other hand, CVaR is the average of tail losses while VaR only serves as a threshold of large loss. Therefore, CVaR may provide incentives for risk managers to take into account tail risks beyond VaR.  Suppose that the loss is a random function of some parameters. The VaR and CVaR of the loss are therefore functions of the parameters. The partial derivatives of these function are called the sensitivities of VaR and CVaR \citep{hong2009estimating,hong:liu:2009,fu2009conditional,jiang2015estimating}. These sensitivities are useful in risk management and gradient-based optimization algorithms. Moreover, \cite{asim:2019} pointed out that sensitivity analysis and capital allocation problems both boil down to similar mathematical formulations. The Euler allocation rule of the total
regulatory capital set via CVaR is a special case of CVaR sensitivities.
In this paper, we focus on sensitivity analysis of CVaR using simulation.  

Sensitivity estimation has been studied extensively in financial engineering. It includes the sensitivities of option prices, which are known as Greeks \cite[see][]{broa:glass:1996}. Additionally, \cite{hong2014estimating} considered the problem of estimating the  sensitivities of portfolio credit risk.  There are several widely used approaches in the simulation literature. The finite difference (FD) method is the simplest one, but it suffers from the trade-off between the bias and variance of the estimator \cite[see][]{fox:glynn:1989}. The infinitesimal perturbation analysis (IPA) takes the pathwise derivatives of the performance function in the estimation, which is also known as the pathwise method \cite[see][]{glas:2004}. From another perspective, the  performance function is viewed as a parameter-free function of some random variables whose joint distribution depends on the parameter. The likelihood ratio (LR) method takes the derivatives of the joint density in the estimation.
Both IPA and LR methods enjoy the usual Monte Carlo variance rate $O(1/n)$, which is faster than the FD method. Unlike LR, IPA needs stronger conditions that rule out discontinuous functions. However, the IPA estimator usually has a smaller variance than the LR estimator when they are both applicable \citep{cui2019variance}. Recently, \cite{peng2018new} proposed a generalized LR method that extends IPA and LR to handle discontinuous functions.

Related work on sensitivity analysis of CVaR includes \cite{scail:2004} and \cite{hong:liu:2009}. Particularly, \cite{hong:liu:2009} proposed an IPA type CVaR sensitivity estimator. Under certain conditions,  the IPA method is  applicable because the performance function of CVaR is a  hockey stick function. In addition,  \cite{hong:liu:2009} established a central limit theorem for the proposed estimator. It should be noted that the IPA estimator of CVaR sensitivity is biased since it is a sample average of dependent observations that includes the VaR estimator, differently from the IPA estimators developed in Greeks estimation. In this paper, we analyze the IPA estimator proposed by \cite{hong:liu:2009} in the framework of randomized quasi-Monte Carlo (RQMC).
 RQMC is a randomized version of quasi-Monte Carlo (QMC). The (R)QMC  method which has the potential to accelerate the
convergence becomes an alternative method in simulation. It is widely used in financial engineering, such as option pricing and Greeks estimation \citep[see, e.g.,][]{joy1996quasi,wang2013pricing,xie2019importance}. Recently, \cite{he_convergence_2017} established a deterministic error bound for the QMC-based quantile estimator. They also showed that under certain conditions the RQMC-based CVaR estimator has a root mean squared error (RMSE) of $O(n^{-1/2-1/(4d-2)+\epsilon})$ for arbitrarily small $\epsilon>0$, where $d$ is the dimension of RQMC points used in the simulation. To the best of our knowledge, very few works are concerned with sensitivity estimation of CVaR in (R)QMC. The results in \cite{he_convergence_2017} cannot be extended to sensitivity estimation. Particularly, the strong consistency for the VaR and CVaR estimators is remain unclear when using RQMC.

There are two major difficulties in analyzing the RQMC-based CVaR sensitivity estimator. The first difficulty is that the estimator is an estimated function at an estimated VaR rather than the usual sample average. As a result, the numerical analysis in (R)QMC quadrature cannot be applied directly. The second difficulty is due to the discontinuity and singularity of the performance function of CVaR sensitivity. For this case, the  well-known Koksma-Hlawka inequality for assessing QMC error is useless \citep{nied:1992}.  In this paper, we first establish the strong consistency for the RQMC-based estimator by making use of the recent work of \cite{owen:2020}. As a by-product, the strong consistency for the VaR and CVaR estimators is also proved under very mild conditions.   We then study the convergence rates of the estimator. Under some technical conditions, we find that the RQMC-based estimator yields a mean error rate of $O(n^{-1/2-1/(4d-2)+\epsilon})$ for arbitrarily small $\epsilon>0$. The technical conditions are verified for portfolio loss of geometric Brownian motions driven models or when the loss is a quadratic form of normally distributed variables.   Our contribution is two-fold. 
\begin{itemize}
	\item We give a rigorous error analysis of CVaR sensitivity estimation when using RQMC. It is sometimes straightforward to replace Monte Carlo with RQMC for practical problems in finance. Such a simple replacement often leads to an improvement as frequently observed in the numerical results. Theoretical analysis is expected to better understand the performance of RQMC. 
	\item As a by-product of the error analysis, we find that the efficiency of CVaR sensitivity estimation depends on the RQMC integration of two specific discontinuous integrands. This paves the way to improve the RQMC accuracy in CVaR sensitivity estimation. If some strategies are applied to improve the RQMC integration of the two integrands, one would expect a better performance of CVaR sensitivity estimation. In the (R)QMC literature, there are some promising strategies to handle discontinuities in numerical integration, such as dimension reduction techniques and smoothing methods \citep{wang2013pricing,zhang2019quasi}.
\end{itemize}

The remainder of this paper is organized as follows. In Section~\ref{sec:intro}, we present some background on VaR, CVaR and its sensitivity estimation. In Section~\ref{sec:main}, we focus on analyzing the RQMC-based CVaR sensitivity estimator.
Some important (R)QMC preliminaries are first reviewed. Strong consistency and some stochastic bounds are then established. Section~\ref{sec:examples} gives some numerical examples, including portfolios of European
options modeled to be driven by a geometric Brownian motion  and a quadratic loss model arising from the delta-gamma approximation of portfolio value change. Section~\ref{sec:concl} concludes this paper. A technical lemma and its proof are deferred to the Appendix.

\section{Background and Simulation-based Estimation}\label{sec:intro}
Let $L$ be the random loss and $F_L(y)=\mbp(L\le y)$ be the cumulative distribution function (CDF) of $L$. For any $\alpha\in (0,1)$, we define, respectively, the $\alpha$-VaR and  the $\alpha$-CVaR of $L$ as 
\begin{align*}
v_\alpha &= F_L^{-1}(\alpha)=\inf\{y\in \mathbb{R}:F_L(y)\ge \alpha\},\\
c_\alpha &=\frac 1{1-\alpha}\int_{\alpha}^{1}v_\beta \mrd \beta=v_\alpha+\frac 1{1-\alpha}\E{(L-v_\alpha)^+},
\end{align*}
where $(a)^+=\max(a,0)$. The $\alpha$-VaR is the lower $\alpha$-quantile of the distribution of $L$. If $L$ has a density in a neighborhood of $v_\alpha$, then $c_\alpha=\mbe[L|L\ge v_\alpha]$. From this point of view, CVaR is the expected shortfall or the tail conditional expectation.

Suppose that $L_1,\dots,L_n$ are $n$ observations of $L$. Monte Carlo sampling renders independent and identically distributed (iid) observations. In the QMC sampling, the observations are deterministic, but with better uniformness.
In the RQMC framework, the observations are dependent random variables retaining the better uniformness (see Section~\ref{sec:qmc} for the details). In either case, based on the $n$ observations, one can estimate the CDF $F_L(y)$ by the so-called empirical CDF
\begin{equation}\label{eq:ecdf}
\hat{F}_n(y)=\frac{1}{n}\sum_{i=1}^{n}\bm{1}\{L_i\le y\}.
\end{equation}
It is natural to estimate the $\alpha$-VaR and  the $\alpha$-CVaR by
\begin{align}
\hat{v}_{\alpha,n} &= \hat F_n^{-1}(\alpha)=\inf\{y\in \mathbb{R}:\hat F_n(y)\ge \alpha\},\label{eq:varest}\\
\hat{c}_{\alpha,n} &=\hat{v}_{\alpha,n}+\frac 1{n(1-\alpha)}\sum_{i=1}^{n}(L_i-\hat{v}_{\alpha,n})^+,\label{eq:cvarest}
\end{align}
respectively. The crucial step is to  estimate the CDF $F_L(y)$ via \eqref{eq:ecdf}. We expect that a better estimate of the CDF leads to an improved estimate of the quantile. 
Recently, \cite{kaplan:2019} compared two approaches for quantile estimation via RQMC.

Suppose that the random loss can be modeled as a function $L(\theta)$, where $\theta$ is the parameter of interest with range $\Theta\subset \mathbb{R}$. To emphasize the dependence on the parameter $\theta$, the $\alpha$-VaR and  the $\alpha$-CVaR of $L(\theta)$ are rewritten as $v_\alpha(\theta)$ and $c_\alpha(\theta)$, respectively. In this paper, we are interested in the sensitivity estimation of CVaR  with respect to $\theta$, i.e., $c'_\alpha(\theta)=\mrd c_\alpha(\theta)/\mrd \theta$. Let $L'(\theta) = \mrd L(\theta)/\mrd \theta$. To obtain an IPA estimator of the CVaR sensitivity, we need the following technical assumptions.

\begin{assumption}\label{assum:lips}
	There exists a random variable $K$ with $\E{K}<\infty$ such that $$\abs{L(\theta_2)-L(\theta_1)}\le K\abs{\theta_2-\theta_1}$$ for all $\theta_1,\theta_2\in \Theta$, and $L'(\theta)$ exists with probability 1 (w.p.1) for all $\theta \in \Theta$. 
\end{assumption}

\begin{assumption}\label{assum:diff}
	The VaR function $v_\alpha(\theta)$ is differentiable for any $\theta \in \Theta$.
\end{assumption}

\begin{assumption}\label{assum:zeroprob}
	For any $\theta \in \Theta$, $\mbp [L(\theta)=v_\alpha(\theta)]=0$.
\end{assumption}

Under Assumptions \ref{assum:lips}--\ref{assum:zeroprob}, \cite{hong:liu:2009} showed that by interchanging the expectation and differentiation,
\begin{align*}
c'_\alpha(\theta)&=v'_\alpha(\theta)+\frac{1}{1-\alpha}\mbe\left[\frac{\partial (L(\theta)-v_\alpha(\theta))^+}{\partial\theta}\right]\notag\\
&=v'_\alpha(\theta)+\frac{1}{1-\alpha}\mbe\left[ L'(\theta)\bm 1\{L(\theta)> v_\alpha(\theta)\}-v'_\alpha(\theta)\bm 1\{L(\theta)> v_\alpha(\theta)\}\right]\notag\\
&=\frac 1{1-\alpha}\E{L'(\theta)\bm 1\{L(\theta)> v_\alpha(\theta)\}}.
\end{align*}
To simplify the notation, we let $L$ and $L'$ denote $L(\theta)$ and $L'(\theta)$, respectively.  
Suppose that we are able to obtain $n$ observations of $(L,L')$, denoted by $(L_1,L'_1),\dots,(L_n,L'_n)$. \cite{hong:liu:2009}  proposed an IPA estimate of $c'_\alpha(\theta)$ given by
\begin{equation}\label{eq:est}
\hat{\mu}_n=\frac{1}{n(1-\alpha)}\sum_{i=1}^nL'_i\bm 1\{L_i> \hat{v}_{\alpha,n} \},
\end{equation}
where $\hat{v}_{\alpha,n}$ is obtained by \eqref{eq:varest}. The IPA estimator is a sample average of dependent variables that involve the VaR estimator $\hat{v}_{\alpha,n}$. As a result, the CVaR sensitivity estimator \eqref{eq:est} is biased. \cite{hong:liu:2009} proved an asymptotic bias of $o(n^{-1/2})$ in the Monte Carlo setting.

The random loss is often expressed as a function of random variables, say, $L(\theta)=g_\theta(\bm \omega)$ for $\bm \omega=(\omega_1,\dots,\omega_s)^\top $. The random variables $\omega_i$ are  called the risk factors or scenarios. In this paper, we assume that the closed forms of $g_\theta$ and $\partial g_\theta/\partial \theta$ are available so that one can generate the sample $(L_i,L'_i)$. If the random loss is modeled as a conditional expectation $L(\theta)=\mbe[Z|\bm \omega]$ which is not given analytically, one may resort to nested simulation. This is beyond the scope of this paper. Interested readers are referred to 
\cite{gordy_nested_2010,broadie_efficient_2011}. 

\section{RQMC-based Estimation of CVaR Sensitivity}\label{sec:main}

The incorporation of QMC or RQMC
in  estimating the CVaR sensitivity $c'_\alpha(\theta)$ is straightforward provided that the mechanism of sampling the random loss $L$ via standard uniform distributed variables is specified. In what follows, we suppose that $L$ can be expressed as
\begin{equation}\label{eq:model}
L=g_\theta(\bm u) = g_\theta(u_1,\dots,u_d),
\end{equation}
where $g_\theta:(0,1)^d\to \mathbb{R}$ is a given measurable function, and $\bm u=(u_1,\dots,u_d)^\top\sim\Unif{(0,1)^d}$. The model \eqref{eq:model} was also studied in VaR estimation with Latin hypercube sampling \citep{avramidis_correlation-induction_1998,dong_quantile_2017}. It is allowed that $g_\theta(\bm u)$ has singularities along the boundary of the unit cube.   That is why we do not consider the closed set $[0,1]^d$ or the half closed set $[0,1)^d$. As we can see from the numerical examples in Section~\ref{sec:examples}, the random loss is unbounded with singularities along the boundary of the unit cube.  Let $g'_\theta(\bm u)=\partial g_\theta(\bm u)/\partial \theta$. Instead of generating an iid sample in the Monte Carlo framework, we now generate the sample
\begin{equation*}
L_i =  g_\theta(\bm u_i),\ L'_i = g'_\theta(\bm u_i),\ i=1,\dots,n,
\end{equation*}
via RQMC points $\bm u_i$ in $(0,1)^d$. It then follows by substituting the sample in \eqref{eq:ecdf} and \eqref{eq:est} to obtain the associated CVaR sensitivity estimator.  There are various QMC points designed with high uniformness in the literature, which are known as low discrepancy points. A better performance of RQMC can be therefore expected. The price for switching from Monte Carlo to (R)QMC is to specify the model as the form
\eqref{eq:model}.

\subsection{QMC and RQMC theory}\label{sec:qmc}

In this subsection, we review the philosophy of the QMC world and some important QMC integration error analysis in the literature. Let's start by estimating an integral over the unit cube 

\begin{equation*} 
I(f) = \mbe[f(\bm{u})]= \int_{(0,1)^d}f(\bm{x})\mrd \bm{x}.
\end{equation*}
The QMC quadrature rule takes the average
\begin{equation}\label{eq:deterestimate}
\hat{I}_n(f)= \frac 1 n \sum_{i=1}^n f(\bm{u}_i),
\end{equation}
where $\bm{u}_1,\dots,\bm{u}_n$ are carefully chosen points in $(0,1)^d$.  The well-known Koksma-Hlawka inequality gives a deterministic error bound for the quadrature rule \eqref{eq:deterestimate},
\begin{equation}\label{K-H}
\abs{\hat{I}_n(f)-I(f)}\leq V_{\BV}(f)D^{*}_n(\mathcal{P}),
\end{equation}
where $\mathcal{P}:=\{\bm{u}_1,\dots,\bm{u}_n\}$, $V_{\BV}(f)$ is the variation of $f(\bm{u})$ in the sense of Hardy and Krause, and $D^{*}_n(\mathcal{P})$ is the star-discrepancy of points in $\mathcal{P}$; see \cite{nied:1992} for details. There are many ways to construct low discrepancy point sets such that $D^{*}_n(\mathcal{P})=O(n^{-1}(\log n)^d)$.  By \eqref{K-H}, the QMC error is of $O(n^{-1}(\log n)^d)$ for integrands with bounded variation in the sense of Hardy and Krause (BVHK). However, if the integrand $f$ is discontinuous or unbounded, the variation is usually unbounded \citep{owen:2005}. For this case, the Koksma-Hlawka inequality is useless. In this paper, we restrict our attention to $(t,m,d)$-nets in base $b\geq 2$ for which the sample size has the form $n=b^m$. Our results also work for
$(t,d)$-sequences without that restriction on $n$. For $p\ge 1$, the space $L^p((0,1)^d)$ consists of all measurable function $f$ on $(0,1)^d$ for which $\int_{(0,1)^d} f(\bm u)^p \mrd \bm u  <\infty$.

\begin{definition}\label{defn1}
	An elementary interval in base $b$ is a subset of $[0,1)^d$ of the form
	\begin{equation*}\label{eq:EI}
	E= \prod_{j=1}^d\bigg[\frac{t_j}{b^{k_j}},\frac{t_j+1}{b^{k_j}}\bigg),
	\end{equation*}
	where $k_j\in \nat $, $t_j\in \nat $ with $t_j<b^{k_j}$ for $j=1,\dots,d$.
\end{definition}

\begin{definition}\label{defnnets}
	Let $t$ and $m$ be nonnegative integers with $t\leq m$. A finite sequence $\bm{u}_1,...,\bm{u}_{b^m} \in [0,1)^d$ is a $(t,m,d)$-net in base $b$ if every elementary interval in base $b$ of volume $b^{t-m}$ contains exactly $b^t$ points of the sequence.
\end{definition}
\begin{definition}
	Let $t$ be a nonnegative integer. An infinite sequence $\bm u_i\in [0,1)^d$ is a $(t,d)$-sequence in base $b$ if for all $k\geq 0$ and $m\geq t$ the finite sequence $\bm{u}_{kb^m+1},...,\bm{u}_{(k+1)b^m}$ is a $(t,m,d)$-net in base $b$.
\end{definition}

For deterministic QMC, it is important to obtain an estimate of the quadrature error $|\hat{I}_n(f)-I(f)|$. But the upper bound in \eqref{K-H} is very hard to compute, and it is restricted to functions of finite variation. Instead, one can randomize the  points $\bm{u}_1,\dots,\bm{u}_n$ and treat the random version of the quadrature $\hat{I}_n(f)$ in \eqref{eq:deterestimate} as an RQMC quadrature rule. That is why we focus on using RQMC.
Usually, the randomized points are uniformly distributed over $(0,1)^d$, and the low discrepancy property of the points is preserved under the randomization (see \cite{lecu:lemi:2005} and Chapter 13 of the monograph \cite{dick:2010}  for a survey of various RQMC methods). In this paper, we focus on the use of scrambling technique proposed by \cite{owen:1995} to randomize  $(t,m,d)$-nets. 

\cite{owen:1995} applied a scrambling scheme on the nets that retains the net property. Let $\bm{u}_i=(u_i^1,\dots,u_i^d)$. We may write the components of $\bm{u}_i$ in their base $b$ expansion $u_i^j = \sum_{k=1}^{\infty} a_{ijk}b^{-k},$
where $a_{ijk}\in\{0,\dots,b-1\}$ for all $i,j,k$. The scrambled version of $\bm{u}_1,\dots,\bm{u}_n$ is a sequence $\tilde{\bm{u}}_1,\dots,\tilde{\bm{u}}_n$ with  $\tilde{\bm{u}}_i=(\tilde{u}_i^1,\dots,\tilde{u}_i^d)$ written as $\tilde{u}_i^j = \sum_{k=1}^{\infty}\tilde{a}_{ijk}b^{-k},$ where $\tilde{a}_{ijk}$ are defined in terms of random permutations of the $a_{ijk}$. The permutation applied to $a_{ijk}$ depends on the values of $a_{ijh}$ for $h=1,\dots,k-1$. Specifically, $\tilde{a}_{ij1} = \pi_j(a_{ij1}),\ \tilde{a}_{ij2} = \pi_{ja_{ij1}}(a_{ij2}),\ \tilde{a}_{ij3} = \pi_{ja_{ij1}a_{ij2}}(a_{ij3})$, and in general
\begin{equation}\label{eq:scrambling}
\tilde{a}_{ijk} = \pi_{ja_{ij1}a_{ij2}\dots  a_{ijk-1} }(a_{ijk}).
\end{equation}
Each permutation $\pi_\bullet$ is uniformly distributed over the $b!$ permutations of $\{0,\dots,b-1\}$, and the permutations are mutually independent. There are some good properties  of scrambled digital nets or sequences, which can be found in \cite{owen:1995,owen:1997b}.

\begin{itemize}
	\item A scrambeled $(t,m,d)$-net and scrambled $(t,d)$-sequence are $(t,m,d)$-net and  $(t,d)$-sequence w.p.1, respectively.
	\item For any point in $(0,1)^d$, the scrambling version of the point is uniformly distributed over  $(0,1)^d$. This implies that the estimate \eqref{eq:deterestimate} is unbiased if using the scrambling method to randomize the QMC points.
	\item If $\bm u_i$ in \eqref{eq:deterestimate} are points of  a scrambled  $(t,m,d)$-net, then for any squared integrable integrand $f$, $\var{\hat I_n(f)}=o(1/n)$. This suggests that the RQMC-based estimate is asymptotically faster than Monte Carlo estimates for a large class of integrands.
\end{itemize}

The results on scrambled nets are highly dependent on the smoothness properties of
the integrand. If the integrand is sufficiently smooth, the scrambled net variance is improved to $O(n^{-3}(\log n)^{d-1})$; see \cite{owen_scrambled_1997, owen_local_2008} for details. On the other hand, if the integrand is discontinuous,  the scrambled net variance turns out to be $O(n^{-1-1/(2d-1)+\epsilon})$ for arbitrarily small $\epsilon>0$ \citep{he:wang:2015,he:2018}. We next encapsulate their results as a proposition that will be used in the following error analysis.

\begin{proposition}\label{prop:he}
	Let $f(\bm u)=g(\bm u)\bm 1\{\bm u\in \Omega\}$, where $\Omega\subset(0,1)^d$ and $g\in L^2((0,1)^d)$. Suppose that $\hat{I}_n(f)$ given by \eqref{eq:deterestimate} is an RQMC quadrature rule using a scrambled  $(t,m,d)$-net in base $b\ge 2$ with $n=b^m$. Assume that the boundary of the set $\Omega$ admits a $(d-1)$-dimensional Minkowski content $\mathcal{M}(\partial\Omega)$, defined by 
	\begin{equation}\label{eq:minkowski}
	\mathcal{M}(\partial\Omega):=\lim_{\epsilon\to 0} \frac{\lambda_d((\partial \Omega)_\epsilon)}{2\epsilon}	<\infty,
	\end{equation}
	where $\lambda_d(\cdot)$ is the $d$-dimensional Lebesgue measure, and $(A)_\epsilon$ denotes the outer parallel body of $A$ at distance $\epsilon$. 
	\begin{itemize}
		\item If $g$ is constant, then $\var{\hat{I}_n(f)}=O(n^{-1-1/d})$. 
		\item If $g$ is of BVHK, then $\var{\hat{I}_n(f)}=O(n^{-1-1/(2d-1)+\epsilon})$ for arbitrarily small $\epsilon>0$. 
		\item If $g$ satisfies the boundary growth condition with arbitrarily small rates (see Definition~\ref{defn:grow}), then $\mbe[|\hat{I}_n(f)-I(f)|]=O(n^{-1/2-1/(4d-2)+\epsilon})$  for arbitrarily small $\epsilon>0$.
	\end{itemize}

\end{proposition}
\begin{proof}
The first and second parts can be found in Theorem 4.4 and Theorem 3.5 of \cite{he:wang:2015}, respectively. The last part is given in Corollary 3.5 of \cite{he:2018}.
\end{proof}

The use of $\epsilon$ in the convergence orders  is to hide the logarithmic factor. As commented in \cite{he:wang:2015}, $\mathcal{M}(\partial \Omega)$  is the surface area of the set $\Omega$ in the terminology of geometry. If $\Omega$ is a convex set in $(0,1)^d$, then its boundary has a $(d-1)$-dimensional Minkowski content. In this case, $\mathcal{M}(\partial \Omega)\leq 2d$ since the surface area of a convex set  in $(0,1)^d$ is bounded by that of  the unit cube, which is $2d$. More generally, \cite{Ambrosio2008} found that if $\Omega$ has a Lipschitz boundary, then $\partial \Omega$ admits a ($d-1$)-dimensional Minkowski content. 

It should be noted that the true value $v_\alpha$ is usually unknown. The CVaR sensitivity estimate \eqref{eq:est} is not the usual quadrature rule  of the form \eqref{eq:deterestimate}. But if we replace the VaR estimate $\hat{v}_{\alpha,n}$ with the true value $v_\alpha$ in \eqref{eq:est}, it turns out to be a quadrature rule for the discontinuous function $g'_\theta(\bm u)\bm 1\{g_\theta(\bm u)> v_\alpha\}/(1-\alpha)$. From this point of view, the results in Proposition~\ref{prop:he} can be used to study the error rate of the CVaR sensitivity estimate. The challenge is how to bound the gap due to the replacement. This is the topic of Section~\ref{sec:bonds}.

\subsection{Strong Consistency}

\cite{hong:liu:2009} established the strong consistency for the Monte Carlo sensitivity estimator. Their proof relies heavily on the strong law of large numbers (SLLN) for an iid sample. Recently, \cite{owen:2020} proved the SLLN for scrambled digital net integration for integrands in $L^{p+1}((0,1)^d)$ for any $p>1$. Together with this fundamental result,  the strong consistency of an RQMC-based CVaR sensitivity estimate can be easily proved following the steps in the proof of \citep[][Theorem 4.1]{hong:liu:2009}. Define
\begin{equation}\label{eq:tildemu}
\tilde{\mu}_n(y) :=\frac{1}{n(1-\alpha)}\sum_{i=1}^nL'_i\bm 1\{L_i> y\},
\end{equation}
where $L'_i=g'_\theta(\bm u_i)$ and $L_i=g_\theta(\bm u_i)$. 

In \cite{he_convergence_2017}, the consistency was proved for the VaR  estimator based on deterministic QMC, but not for its randomized counterpart. We are ready to show the strong consistency of the RQMC-based  estimators. We need the following assumption, which is  the minimal requirement for establishing the strong consistency of the Monte Carlo quantile estimator \citep[][p. 75]{serf:1980}.

\begin{assumption}\label{assum:stand}
	For any $\theta \in \Theta$, $v_\alpha(\theta)$ is the unique solution $x$ of $F(x-)\le p\le F(x)$.
\end{assumption}

\begin{theorem}\label{thm:varconsistency}
If Assumption~\ref{assum:stand} holds and the VaR estimator $\hat{v}_{\alpha,n}$ is based on a scrambled $(t,m,d)$-net in base $b\ge 2$ with $n=b^m$, then 
$$\mbp\left(\lim_{n\to\infty}\hat{v}_{\alpha,n} = v_{\alpha}\right)=1.$$
\end{theorem}

\begin{proof}
By the SLLN established in \citep[][Theorem 5]{owen:2020}, for all $x\in\mathbb{R}$,
$$\mbp\left(\lim_{n\to\infty}\hat F_n(x)=F(x)\right)=1.$$
The following steps are in lines with the proof of the strong consistency of the Monte Carlo estimate \citep[][p. 75]{serf:1980}. By Assumption~\ref{assum:stand}, we have $F(v_\alpha-\epsilon)<\alpha<F(v_\alpha+\epsilon)$ for any $\epsilon>0$. Since $\hat F_n(v_\alpha+\epsilon)\to F(v_\alpha+\epsilon)$ and $\hat F_n(v_\alpha-\epsilon)\to F(v_\alpha-\epsilon)$ w.p.1. We thus have 
$$\lim_{n\to \infty} \mbp\left(\cap_{m=n}^\infty \{\hat F_n(v_\alpha-\epsilon)<\alpha<\hat F_n(v_\alpha+\epsilon) \} \right)=1.$$
By \eqref{eq:varest}, we have $\{\hat F_n(v_\alpha-\epsilon)<\alpha<\hat F_n(v_\alpha+\epsilon) \}\subset \{ v_\alpha-\epsilon\le\hat{v}_{\alpha,n}\le v_\alpha+\epsilon \}$. Therefore, 
$$\lim_{n\to \infty} \mbp\left(\cap_{m=n}^\infty \{ |\hat{v}_{\alpha,n}- v_\alpha|\le \epsilon \} \right)=1,$$
implying $\hat{v}_{\alpha,n} \to  v_{\alpha}$ w.p.1.
\end{proof}

\begin{theorem}\label{thm:consistency}
	Suppose that Assumptions~\ref{assum:lips}--\ref{assum:zeroprob} and \ref{assum:stand} are satisfied. The CVaR sensitivity estimator $\hat{\mu}_n$ given by \eqref{eq:est} is based on a scrambled $(t,m,d)$-net in base $b\ge 2$ with $n=b^m$. 
	\begin{itemize}
		\item If the random loss $L=g_\theta(\bm u)\in L^{1+\gamma_1}((0,1)^d)$  for some $\gamma_1>0$, then $\hat{c}_{\alpha,n}\to c_\alpha(\theta)$ w.p.1 as $n\to \infty$.
		\item If $g'_\theta(\bm u)\in L^{1+\gamma_2}((0,1)^d)$  for some $\gamma_2>0$, then $\hat{\mu}_n\to c'_\alpha(\theta)$ w.p.1 as $n\to \infty$.
	\end{itemize}
\end{theorem}
\begin{proof}
	Since $L=g_\theta(\bm u)\in L^{1+\gamma_1}((0,1)^d)$, $L-v_{\alpha}\in L^{1+\gamma_1}((0,1)^d)$. By the SLLN established in \cite{owen:2020}, 
$$\nu_n:=\frac 1{n}\sum_{i=1}^{n}(L_i-v_{\alpha})^+\to\mbe[(L-v_{\alpha})^+]\text{ w.p.1}.$$
By the triangle inequality and Theorem~\ref{thm:varconsistency}, we have
\begin{align*}
\left\lvert\frac 1{n}\sum_{i=1}^{n}(L_i-\hat{v}_{\alpha,n})^+-\mbe[(L-v_{\alpha})^+]\right\rvert&\le \frac 1{n}\sum_{i=1}^{n}\abs{(L_i-\hat{v}_{\alpha,n})^+-(L_i-v_{\alpha})^+}\\&\quad  +\abs{\nu_n-\mbe[(L-v_{\alpha})^+]}\\
&\le \abs{\hat{v}_{\alpha,n}-v_{\alpha}}+\abs{\nu_n-\mbe[(L-v_{\alpha})^+]}\to 0
\end{align*}	
as $n\to \infty$ w.p.1. Together with \eqref{eq:cvarest}, we find that $\hat{c}_{\alpha,n}\to c_\alpha(\theta)$ w.p.1 as $n\to \infty$.

Note that $\tilde{\mu}_n(y)$ given by \eqref{eq:tildemu} is the quadrature rule $\hat I_n(f)$ with $$f(\bm u;y)=\frac{1}{1-\alpha}g'_\theta(\bm u)\bm 1\{g_\theta(\bm u)> y\}\in L^{1+\gamma}((0,1)^d),$$ and $\hat\mu_n=\tilde{\mu}_n(\hat{v}_{\alpha,n})$.  Since $f(\bm u;y)\in L^{1+\gamma}((0,1)^d)$ for some $\gamma>0$, by the SLLN again, 
$$\tilde{\mu}_n(v_{\alpha})\to\E{f(\bm u;v_{\alpha})}=c'_\alpha(\theta)$$ as $n\to \infty$  w.p.1. It suffices to prove that $\tilde{\mu}_n(\hat{v}_{\alpha,n})-\tilde{\mu}_n(v_{\alpha})\to 0$ w.p.1.

Let $D_n=\frac 1 n\sum_{i=1}^n\abs{L'_i}^{1+\gamma}$, where $0<\gamma<\gamma_2$. Since $\abs{g'_\theta(\bm u)}^{1+\gamma}\in L^{1+\gamma'}((0,1)^d)$ for $\gamma'=(\gamma_2-\gamma)/(1+\gamma)>0$, using the SLLN again, $$\mbp(\lim_{n\to\infty}D_n = \E{\abs{g'_\theta(\bm u)}^{1+\gamma}})=1.$$  By the H\"older inequality, we have
\begin{align}
	(1-\alpha)\abs{\tilde{\mu}_n(\hat{v}_{\alpha,n})-\tilde{\mu}_n(v_{\alpha})}&\le D_n^{\frac 1 {1+\gamma}}\left[\frac 1n \sum_{i=1}^n\abs{\bm 1\{L_i> \hat{v}_{\alpha,n}\}-\bm 1\{L_i> v_{\alpha}\}}^{1+\frac 1 \gamma}\right]^{\frac \gamma {1+\gamma}}\notag\\
	&= D_n^{\frac 1 {1+\gamma}}\left[\frac 1n \sum_{i=1}^n\abs{\bm 1\{L_i> \hat{v}_{\alpha,n}\}-\bm 1\{L_i> v_{\alpha}\}}\right]^{\frac \gamma {1+\gamma}}\notag\\
	&=D_n^{\frac 1 {1+\gamma}}\abs{\frac 1n \sum_{i=1}^n(\bm 1\{L_i> \hat{v}_{\alpha,n}\}-\bm 1\{L_i> v_{\alpha}\})}^{\frac \gamma {1+\gamma}}\notag\\
	&=D_n^{\frac 1 {1+\gamma}}\abs{\hat{F}_n(\hat{v}_{\alpha,n})-\hat{F}_n(v_\alpha)}^{\frac \gamma {1+\gamma}},\label{eq:tmu}
\end{align}
where we use the fact that the function $\bm 1\{x>\hat{v}_{\alpha,n}\}-\bm 1\{x>v_{\alpha}\}$ never change the sign when varying $x$.
	
By the definition of $\hat{v}_{\alpha,n}$, we have $\hat{F}_n(\hat{v}_{\alpha,n})\ge \alpha$. Assumption \ref{assum:zeroprob} implies $F(v_\alpha)=\alpha$. Since $\bm 1\{g_\theta(\bm u)\le v_\alpha\}\in L^2((0,1)^d)$, by SLLN, $\hat{F}_n(v_\alpha)\to F(v_\alpha)=\alpha$ w.p.1. As a result,
	$$\liminf_{n\to\infty} \hat{F}_n(\hat{v}_{\alpha,n})-\hat{F}_n(v_\alpha)\ge 0\text{ w.p.1.}$$
	By Theorem~\ref{thm:varconsistency}, $\hat{v}_{\alpha,n}\to v_{\alpha}$ w.p.1. For any $\epsilon>0$, there exists  $N\ge 1$ such that for any $n\ge N$, $\hat{v}_{\alpha,n}<v_\alpha+\epsilon$. So $\hat{F}_n(\hat{v}_{\alpha,n})-\hat{F}_n(v_\alpha)\le \hat{F}_n(v_\alpha+\epsilon)-\hat{F}_n(v_\alpha)\to F(v_\alpha+\epsilon)-F(v_\alpha)=\mbp[v_\alpha<L\le v_\alpha+\epsilon]$ w.p.1. Letting $\epsilon\to 0$, we have $\mbp[v_\alpha<L\le v_\alpha+\epsilon]\to 0$, giving
	$$\limsup_{n\to\infty} \hat{F}_n(\hat{v}_{\alpha,n})-\hat{F}_n(v_\alpha)\le 0\text{ w.p.1.}$$
	This leads to $\lim_{n\to\infty}\hat{F}_n(\hat{v}_{\alpha,n})-\hat{F}_n(v_\alpha)= 0$ w.p.1, completing the proof.	
\end{proof}

\subsection{Stochastic bounds}\label{sec:bonds}

\begin{assumption}\label{assum:crv}
	For any $\theta\in\Theta$, $g_\theta(\bm u)$ defined over $(0,1)^d$ is a continuous random variable whenever $d-1$ components of $\bm u$ are fixed and the remaining one is uniformly distributed over an open interval in $(0,1)$.	
\end{assumption}

\begin{lemma}
	Assume that $\hat{v}_{\alpha,n}$ defined by \eqref{eq:varest} is based on a scrambled $(t,m,d)$-net in base $b\ge 2$ with $n=b^m$. If Assumption~\ref{assum:crv} is satisfied, then for any $n\ge 1$
	\begin{equation}\label{eq:indv}
	\abs{ \hat{F}_n(\hat{v}_{\alpha,n})-\hat{F}_n(v_\alpha)}\le \frac{b^t}{n}+\abs{\hat{F}_n(v_\alpha) -\alpha} \text{ w.p.1}.
	\end{equation}
\end{lemma}
\begin{proof}
Let $L_{(1)},\dots,L_{(n)}$ be the order statistics of $L_1,\dots,L_n$ in increasing order.
By the definition of $\hat{v}_{\alpha,n}$,  we have $\hat{v}_{\alpha,n}=L_{(\ceil{n\alpha})}$, where $\ceil{a}$ denotes the smallest integer no less than $a$. Notice that
\begin{align*}
\hat{F}_n(\hat{v}_{\alpha,n})&=\frac{1}{n}\sum_{i=1}^n \bm 1\{L_{(i)}\le L_{(\ceil{n\alpha})}\}=\frac{\ceil{n\alpha}-1}{n}+\frac{1}{n}\sum_{i=\ceil{n\alpha}}^n \bm 1\{L_{(i)}= L_{(\ceil{n\alpha})}\}.
\end{align*}
Therefore, 
$$\alpha\le \frac{\ceil{n\alpha}}{n}\le \hat{F}_n(\hat{v}_{\alpha,n})\le \alpha+\frac{1}{n}\sum_{i=1}^n \bm 1\{L_{(i)}=L_{(\ceil{n\alpha})}\}.$$
Let $S_n=\sum_{i=1}^n \bm 1\{L_{(i)}=L_{(\ceil{n\alpha})}\}$. It then follows 
$$\abs{ \hat{F}_n(\hat{v}_{\alpha,n})-\hat{F}_n(v_\alpha)}\le \frac{S_n}{n}+\abs{\hat{F}_n(v_\alpha) -\alpha}.$$

If the loss $L$ is a continuous random variable,  $S_n=1$ w.p.1 for iid observations $L_i$. This is because $\mbp(L_i=L_j)=0$ for all $i\ne j$.	Things become complicated for RQMC sampling since the observations $L_i$ are dependent. Under Assumption~\ref{assum:crv}, Lemma~\ref{lem:append} shows that there are at most $b^t$ of $L_1,\dots,L_n$ with equal value  when using  a scrambled $(t,m,d)$-net w.p.1. This implies that $S_n\le b^t$ w.p.1, completing the proof.
\end{proof}

Note that Sobol' sequences are $(t,d)$-sequences in base $b=2$ with $t$ depending on $d$. For this case, the constant $b^t$ in \eqref{eq:indv} can be reduced to $1$ although the value of $t$ may be much larger than $0$; see Remark~\ref{rem:sobol} for a discussion. 

\begin{remark}\label{rem:piecewise}
	Assumption~\ref{assum:crv} is stronger than Assumption~\ref{assum:zeroprob}. To verify Assumption~\ref{assum:crv}, it turns out to look at a function of one-dimensional variable $v\sim \Unif{(a,b)}$ for any $0\le a<b\le 1$, denoted by $h(v)$. Good smoothness of $h(\cdot)$ does not necessarily render  a continuous random variable.
	For example, $h(x)=e^{-1/(x-0.5)^2}$ for $x>0.5$, and  $0$ otherwise. It is not difficult to see that
	$h\in \mathcal{C}^\infty(\mathbb{R})$, but $\mbp[h(v)=0]=\mbp[0<v\le 0.5]>0$ whenever $a<0.5$. For this case, $h(v)$ is not a continuous random variable. This is due to the absence of strict monotonicity. If the function $h(x)$ is strictly monotonic over $(0,1)$ and $h\in \mathcal{C}^1((0,1))$, then $Y=h(v)$ is a continuous random variable with density 
	\begin{equation}\label{eq:fy}
	f_Y(y)=\frac{1}{(b-a)|h'(h^{-1}(y))|}
	\end{equation}
	for $y$ in the support of $Y$.
	This result can be easily extended to the situation in which $h(x)$ is piecewise strictly monotonic. Suppose that $h\in \mathcal{C}^1((0,1))$ and $h'(x)=0$ has countable solutions on $(0,1)$. It is clear that $h(x)$ is  strictly monotonic on an open interval determined by any two successive solutions, in which $h(v)$ has a density of the form \eqref{eq:fy}. The overall density is piecewise. As a result, $Y=h(v)$ is a continuous random variable.
\end{remark}

\begin{theorem}[Bounded case]\label{thm:mainboundD}
Suppose that Assumptions~\ref{assum:lips}--\ref{assum:diff} and \ref{assum:crv}  are satisfied. The estimator $\hat{\mu}_n$ given by \eqref{eq:est} is based on a scrambled $(t,m,d)$-net in base $b\ge 2$ with $n=b^m$. If $g'_\theta(\bm u)$ is bounded, then
\begin{equation*}\label{eq:msebound}
\E{(\hat{\mu}_n-c'_\alpha(\theta))^2}=o(1/n).
\end{equation*}
Let $\Omega=\{\bm u\in(0,1)^d:g_\theta(\bm u)> v_\alpha\}$. If $g'_\theta(\bm u)$ is of BVHK and $\partial \Omega$ admits a $(d-1)$-dimensional Minkowski content defined by \eqref{eq:minkowski}, then  
\begin{equation}\label{eq:msebound2}
\E{(\hat{\mu}_n-c'_\alpha(\theta))^2}=O(n^{-1-1/(2d-1)+\epsilon})
\end{equation}
for arbitrarily small $\epsilon>0$.
\end{theorem}
\begin{proof}
	Suppose that  $g'_\theta(\bm u)$ is bounded by $M$. Then  by \eqref{eq:indv}, similarly to \eqref{eq:tmu}, we have
	\begin{align}
\abs{\tilde{\mu}_n(\hat{v}_{\alpha,n})-\tilde{\mu}_n(v_{\alpha})}&\le \frac 1 {n(1-\alpha)}\sum_{i=1}^n |g'_\theta(\bm u_i)|\abs{\bm 1\{L_i> \hat{v}_{\alpha,n}\}-\bm 1\{L_i> v_{\alpha}\}}\notag\\
&\le \frac{M}{1-\alpha}\abs{\hat{F}_n(\hat{v}_{\alpha,n})-\hat{F}_n(v_\alpha)}\notag\notag\\&\le \frac{b^t  M}{(1-\alpha)n}+\frac{M}{1-\alpha}\abs{\hat{F}_n(v_\alpha) -\alpha}\text{ w.p.1}.\label{eq:tt}
	\end{align}
As a result, 
\begin{align}
\E{(\hat{\mu}_n-c'_\alpha(\theta))^2}&=\E{(\tilde{\mu}_n(\hat{v}_{\alpha,n})-c'_\alpha(\theta))^2}\notag\\&\le \E{(\tilde{\mu}_n(\hat{v}_{\alpha,n})-\tilde{\mu}_n(v_{\alpha}))^2}+\E{(\tilde{\mu}_n(v_{\alpha})-c'_\alpha(\theta))^2}\notag\\
&\le \frac{2b^{2t}M^2}{n^2(1-\alpha)^2}+\frac{2M^2}{(1-\alpha)^2}\var{\hat{F}_n(v_\alpha)}+\var{\tilde{\mu}_n(v_{\alpha})}.\label{eq:mse}
\end{align}

Notice that $\var{\hat{F}_n(v_\alpha)}=o(1/n)$ because $\bm 1\{g_\theta(\bm u)> v_\alpha\}\in L^2((0,1)^d)$. 
Similarly, thanks to $g'_\theta(\bm u)\bm 1\{g_\theta(\bm u)> v_\alpha\}\in L^2((0,1)^d)$, $\var{\tilde{\mu}_n(v_{\alpha})}=o(1/n),$ giving $\E{(\hat{\mu}_n-c'_\alpha(\theta))^2}=o(1/n)$.

If $\partial \Omega$ admits a $(d-1)$-dimensional Minkowski content, by the first part of Proposition~\ref{prop:he}, $\var{\hat{F}_n(v_\alpha)}=O(n^{-1-1/d})$. If $g'_\theta$ is of BVHK, $\var{\tilde{\mu}_n(v_{\alpha})}=O(n^{-1-1/(2d-1)+\epsilon})$ by using the second part of Proposition~\ref{prop:he}. Using the inequality \eqref{eq:mse} completes the proof.
\end{proof}

Theorem~\ref{thm:mainboundD} requires the boundedness of $g'_\theta(\bm u)$, which may not hold in practice. If $g'_\theta(\bm u)$ is unbounded, the inequality \eqref{eq:tt} does not hold. To get  rid of this, we use a truncated version of $g'_\theta(\bm u)$ so that the inequality \eqref{eq:tt} can be applied. We first introduce  the so-called boundary growth condition for controlling the function $g'_\theta(\bm u)$ around the boundaries of the unit cube. Let $1{:}d=\{1,\dots,d\}$. For a set $v\subseteq 1{:}d$, $\partial^{v}h$ denotes the mixed partial derivative
of $h$ taken once with respect to components with indices in $v$. 
\begin{definition}\label{defn:grow}
	A function $h$ defined on $(0,1)^d$ is said to satisfy the boundary growth condition if
	\begin{equation}\label{eq:grow}
	\abs{\partial^{v}h(\bm u)}\leq B\prod_{i\in v}\min(u_i,1-u_i)^{-A_i-1}\prod_{i\notin v}\min(u_i,1-u_i)^{-A_i}
	\end{equation}
	holds for all $\bm u\in(0,1)^d$, some rates $A_i>0$, some $B<\infty$ and all $v\subseteq 1{:}d$.	
\end{definition}
The boundary growth condition is the second growth condition described in \cite{owen_halton_2006}.
We use a region 
\begin{equation*}
K(\epsilon) = \{\bm u\in[0,1]^d|\prod_{1\le i\le d}\min(u_i,1-u_i)\ge \epsilon\}
\end{equation*}
to avoid the singularities for small $\epsilon>0$. We now define an extension $h_\epsilon$ of $h$ from $K(\epsilon)$ to $[0,1]^d$ such that $h_\epsilon(\bm u)=h(\bm u)$ for $\bm u\in K(\epsilon)$. One can extend the function $h_\epsilon$ to the whole unit cube $[0,1]^d$ with some good properties as in \cite{owen_halton_2006}. That is,
\begin{equation}\label{eq:extg}
h_\epsilon(\bm u) = h(\bm{c})+ \sum_{v\neq \varnothing}\int_{[\bm c^v,\bm{u}^v]}\partial^{v}h(\bm z^v{:}\bm c^{-v})\bm 1\{\bm z^v{:}\bm c^{-v}\in K(\epsilon)\}\mrd \bm z^v,
\end{equation}
where $\bm c=(1/2,\dots,1/2)^\top$, $\bm z^v{:}\bm c^{-v}$ denotes the point $\bm y\in[0,1]^d$ with $y_j=z_j$ for $j\in v$ and $y_j=c_j$ for  $j\notin v$.
Taking $h(\bm u)=g'_\theta(\bm u)$, $h_\epsilon(\bm u)$ serves as an approximation of $g'_\theta(\bm u)$, which is of BVHK \citep{owen_halton_2006} and therefore bounded.

\begin{theorem}[Unbounded case]\label{thm:mainunboundD}
	Suppose that Assumptions~\ref{assum:lips}--\ref{assum:diff} and \ref{assum:crv}  are satisfied. The estimator $\hat{\mu}_n$ given by \eqref{eq:est} is based on a scrambled $(t,m,d)$-net in base $b\ge 2$ with $n=b^m$. Let $\Omega=\{\bm u\in (0,1)^d:g_\theta(\bm u)> v_\alpha\}$. If $g'_\theta(\bm u)$ satisfies the boundary growth condition \eqref{eq:grow} with arbitrarily small rates $A_i>0$ and $\partial \Omega$ admits $(d-1)$-dimensional Minkowski content defined by \eqref{eq:minkowski}, then $$\E{\abs{\hat{\mu}_n-c'_\alpha(\theta)}}=O(n^{-1/2-1/(4d-2)+\epsilon})$$	for arbitrarily small $\epsilon>0$.
\end{theorem}
\begin{proof}
	We let $h(\bm u)=g'_\theta(\bm u)$ from now on and work on the extension $h_\epsilon(\bm u)$ defined by \eqref{eq:extg},  where $\epsilon>0$ is to be determined. Let $M_\epsilon=\sup_{\bm u\in[0,1]^d}|h_\epsilon(\bm u)|$. By the triangle inequality and using \eqref{eq:tt} by replacing $g'_\theta(\bm u)$ with $h_\epsilon(\bm u)$, with probability one, 
	\begin{align*}
	\abs{\tilde{\mu}_n(\hat{v}_{\alpha,n})-\tilde{\mu}_n(v_{\alpha})} &= \abs{\frac 1 {n(1-\alpha)}\sum_{i=1}^n h(\bm u_i)(\bm 1\{L_i> \hat{v}_{\alpha,n}\}-\bm 1\{L_i> v_{\alpha}\})}\\
	&\le \abs{\frac 1 {n(1-\alpha)}\sum_{i=1}^n h_\epsilon(\bm u_i)(\bm 1\{L_i> \hat{v}_{\alpha,n}\}-\bm 1\{L_i> v_{\alpha}\})}\\
	&+\abs{\frac 1 {n(1-\alpha)}\sum_{i=1}^n [h(\bm u_i)-h_\epsilon(\bm u_i)](\bm 1\{L_i> \hat{v}_{\alpha,n}\}-\bm 1\{L_i> v_{\alpha}\})}\\
	&\le \frac{b^t  M_\epsilon}{(1-\alpha)n}+\frac{M_\epsilon}{1-\alpha}\abs{\hat{F}_n(v_\alpha) -\alpha}+\frac 1 {n(1-\alpha)}\sum_{i=1}^n \abs{h(\bm u_i)-h_\epsilon(\bm u_i)}.
	\end{align*}
	Taking the expectation, we have
	\begin{equation}\label{eq:soex}
	\mbe[\abs{\tilde{\mu}_n(\hat{v}_{\alpha,n})-\tilde{\mu}_n(v_{\alpha})}]\le \frac{b^t  M_\epsilon}{(1-\alpha)n}+\frac{M_\epsilon}{1-\alpha}\var{\hat{F}_n(v_\alpha)}^{1/2}+\frac{1}{1-\alpha}\mbe[\abs{h(\bm u)-h_\epsilon(\bm u)}],
	\end{equation}
	where we use the fact that each $\bm u_i\sim \Unif{(0,1)^d}$ due  to scrambling.
	
	If $h(\bm u)$ satisfies the boundary growth condition with rates $A_i$, Propositions 3.2 and 3.3 in \cite{he:2018} show that for any $\eta \in(0,1-\max_i A_i)$, there exists $C<\infty$ such that
	\begin{align*}
	\mbe[\abs{h(\bm u)-h_\epsilon(\bm u)}]&\le C \epsilon^{1-\max_i A_i-\eta},\\
	M_\epsilon &\le C \epsilon^{-\max_i A_i-\eta}.
	\end{align*}
	Note that $\var{\hat{F}_n(v_\alpha)}=O(n^{-1-1/d})$ by the first part of Proposition~\ref{prop:he}. Taking $\epsilon \propto n^{-1/2-1/(2d)}$, it then follows from \eqref{eq:soex} that
		$$\mbe[\abs{\tilde{\mu}_n(\hat{v}_{\alpha,n})-\tilde{\mu}_n(v_{\alpha})}]=O(n^{(1-\max_i A_i-\eta)[-1/2-1/(2d)]}).$$
	Since $A_i$ and $\eta$ are arbitrarily small positive numbers, we conclude  $\mbe[\abs{\tilde{\mu}_n(\hat{v}_{\alpha,n})-\tilde{\mu}_n(v_{\alpha})}]=O(n^{-1/2-1/(2d)+\epsilon'})$ for arbitrarily small $\epsilon'>0$. By the last part of Proposition~\ref{prop:he}, we have $\mbe[\abs{\tilde{\mu}_n(v_{\alpha})-c'_\alpha(\theta)}]=O(n^{-1/2-1/(4d-2)+\epsilon'})$. As a result,
	$$\mbe[\abs{\hat\mu_n-c'_\alpha(\theta)}]\le \mbe[\abs{\tilde{\mu}_n(\hat{v}_{\alpha,n})-\tilde{\mu}_n(v_{\alpha})}]+\mbe[\abs{\tilde{\mu}_n(v_{\alpha})-c'_\alpha(\theta)}]=O(n^{-1/2-1/(4d-2)+\epsilon'}).$$
\end{proof}

\begin{remark}
	From the proofs of Theorems~\ref{thm:mainboundD} and \ref{thm:mainunboundD}, the error of the CVaR sensitivity estimator is bounded by the numerical integration errors of the two discontinuous functions: $g'_\theta(\bm u)\bm 1\{g_\theta(\bm u)> v_\alpha\}$ and $\bm 1\{g_\theta(\bm u)\le v_\alpha\}$. We therefore cannot expect a better RQMC error rate than those for RQMC integration with discontinuous integrands. From this point of view, sensitivity estimation of CVaR is rather challenging for RQMC due to the discontinuities and  singularities involved in the two functions. To improve the RQMC efficiency, one should pay more attention to handle the two discontinuous functions.
\end{remark}

\section{Numerical study}\label{sec:examples}

In this section, we examine three cases of CVaR sensitivity estimation. In our numerical experiments on RQMC, we use randomized Sobol' points by the linear scrambling of \cite{mato:1998}, which has been carried out in the generator \verb|scramble| in MATLAB.

\subsection{Single Asset}

Consider a portfolio of a single European put option  that was studied in \cite{broadie_efficient_2011} for testing their adaptive nested simulation. 
The underlying asset follows a
geometric Brownian motion with an initial price of
$S_0=100$. The drift of this process under the real-world
distribution is $\mu = 8\%$. The annualized volatility is $\sigma=20\%$. The
risk-free rate is $r=3\%$. The strike of the put option is
$K = 95$, and the maturity is $T=0.25$ years (i.e., three
months). The risk horizon is $\tau = 1/52$ years (i.e., one
week). 

Denote by $S_\tau $ the underlying asset price at the
risk horizon $\tau$. This price in the real-world is generated
according to
\begin{equation}\label{eq:stau}
S_\tau=S_0\exp\{(\mu-\sigma^2/2)\tau+\sigma\sqrt{\tau}Z\},\ Z=\Phi^{-1}(u),
\end{equation}
where $\Phi$ is the CDF of the standard normal distribution, $u\sim \Unif{(0,1)}$, and hence $Z\sim N(0,1)$. 
Using the Black--Scholes formula \citep{hull:2015}, the value of the put option at time $\tau$ is 

\begin{equation}\label{eq:tauvalue}
v_\tau = \mbe[e^{-r(T-\tau)}(K-S_T)^+| S_\tau]= Ke^{-r(T-\tau)}\Phi(-d_2)-S_\tau\Phi(-d_1),
\end{equation}
where $d_2=d_1-\sigma\sqrt{T-\tau}$, and
$$d_1=\frac{\ln ( S_\tau/K)+(r+\sigma^2/2)(T-\tau)}{\sigma\sqrt{T-\tau}}.$$

The portfolio value loss $L(\theta)= v_0-v_\tau$ can be expressed explicitly as a function of $u$, say, $g_\theta(u)$. Here  $\theta$ is the parameter of interest, such as $S_0$, $\mu$, $r$, $\sigma$ etc. 
For this example, the nested simulation is unnecessary. We now write the initial value $v_0$ as a function of $\theta$, denoted by $v_0(\theta)$. The derivative of $v_0(\theta)$ with respect to $\theta$ (denoted by $v'_0(\theta)$) is known as Greeks \citep[see][Chapter 15]{hull:2015}. From \eqref{eq:tauvalue}, the portfolio value at time $\tau$ can be viewed as a function of the random factor $S_\tau$ and possibly the parameter $\theta$, denoted by $v_\tau(S_\tau,\theta)$. By the chain rule, we have
\begin{align}
L'(\theta) &=g'_\theta(u)= v'_0(\theta) - \frac{\partial v_\tau(S_\tau,\theta)}{\partial S_\tau}\frac{\partial S_\tau}{\partial \theta}-\frac{\partial v_\tau(S_\tau,\theta)}{\partial \theta}\notag\\
&=v'_0(\theta) - [\Phi(d_1)-1]\frac{\partial S_\tau}{\partial \theta}-\frac{\partial v_\tau(S_\tau,\theta)}{\partial \theta},\label{eq:singput}
\end{align}
where we use  the fact that the \verb|delta| of the option at time $\tau$ is $\partial v_\tau(S_\tau,\theta)/\partial S_\tau=\Phi(d_1)-1.$
Assumption~\ref{assum:lips} can be easily verified by taking $\Theta$ as a small neighborhood of the parameter $\theta$ being estimated.

\textbf{Case 1}. Consider $\theta=S_0$. It is easy to see that $\partial S_\tau/\partial \theta = S_\tau/S_0$ and  $\partial v_\tau(S_\tau,\theta)/\partial \theta=0$, and hence $g'_\theta(u)=v'_0(\theta) - [\Phi(d_1)-1]S_\tau/S_0$, which is unbounded. As we will see later, all the assumptions in Theorem~\ref{thm:mainunboundD} are satisfied. The CVaR sensitivity estimate based on RQMC can therefore enjoy a mean error of $O(n^{-1+\epsilon})$ for arbitrarily small $\epsilon>0$.

\textbf{Case 2}. Consider $\theta=r$.  It is easy to see that  $\partial S_\tau/\partial \theta = 0$ and  $\partial v_\tau(S_\tau,\theta)/\partial \theta=-K(T-\tau)e^{-r(T-\tau)}\Phi(-d_2)$. The later is known as the \verb|rho|   of  the option at time $\tau$. So $g'_\theta(u)=v'_0(\theta) +K(T-\tau)e^{-r(T-\tau)}\Phi(-d_2)$, which is bounded. Notice that
\begin{align*}
\frac{\partial g'_\theta(u)}{\partial u}&=-K(T-\tau)e^{-r(T-\tau)}\sqrt{\frac{\tau}{T-\tau}}\frac{\Phi'(-d_2)}{\Phi'(\Phi^{-1}(u))}\\
&=-Ke^{-r(T-\tau)}\sqrt{\tau(T-\tau)}\exp\{[(\Phi^{-1}(u))^2-d_2^2]/2\},
\end{align*}
where $\Phi'(u)$ is the density of the standard normal distribution. Since $d_2 = \sqrt{\tau/(T-\tau)}\Phi^{-1}(u)+C$ for some constant $C$, $\abs{\partial g'_\theta(u)/\partial u}$ goes to infinity as $u\to 0$ or $1$ when $\tau<T-\tau$. It is common that the risk horizon $\tau$ is very small relative to the maturity $T$. So for this case,  $\partial g'_\theta(u)/\partial u$ is unbounded, leading to unbounded variation of $g'_\theta (u)$ in the sense of Hardy and Krause. Therefore, the faster rate \eqref{eq:msebound2} in Theorem~\ref{thm:mainboundD} cannot be applied.  Recall that the mean error can be bounded
by the root mean squared error. Theorem~\ref{thm:mainboundD} directly yields that the mean
error is $o(n^{-1/2})$. This is rather conservative because the problem is only one-dimensional. On the other hand, by applying Theorem~\ref{thm:mainunboundD} which allows unbounded $\partial g'_\theta(u)/\partial u$,  a mean error of $O(n^{-1+\epsilon})$ can be achieved.

We now verify the assumptions in Theorem~\ref{thm:mainunboundD} for various cases of $\theta$ including the two cases above. It is not difficult to see that $\partial g_\theta(u) /\partial u>0$, implying the loss $g_\theta(u)$ is strictly increasing in $u$.
So $\alpha = \mbp(g_\theta(u)\le v_\alpha(\theta))=\mbp(u\le u^*)=u^*$, where $u^*$ is the unique solution to $g_\theta(u)=v_\alpha(\theta)$. The closed-form for $\alpha$-VaR is thus available, i.e., $v_\alpha(\theta) = g_\theta(\alpha)$. It is obvious that Assumptions~\ref{assum:diff} and \ref{assum:zeroprob} and \ref{assum:crv} are satisfied. Note that $\Omega=\{ u\in(0,1):g_\theta(u)> v_\alpha\}=(\alpha,1)$, whose boundary admits Minkowski content \cite{Ambrosio2008}. It remains to show that  $g'_\theta( u)$ satisfies the boundary growth condition \eqref{eq:grow} with an arbitrarily small growth rate. That is, for arbitrarily small $A>0$,
\begin{align*}
\abs{g'_\theta( u)}= O(\min(u,1-u)^{-A}),\ \abs{\frac{\partial g'_\theta(u)}{\partial u}}= O(\min(u,1-u)^{-A-1}).
\end{align*}
To this end, we first introduce some useful upper bounds that were also used in \cite{he:2018,he_error_2019}. Note that 
\begin{equation}\label{eq:lo}
\Phi^{-1}(\epsilon)=-\sqrt{-2\ln(\epsilon)}+o(1),\ \Phi^{-1}(1-\epsilon)=\sqrt{-2\ln(\epsilon)}+o(1)
\end{equation}
as $\epsilon\downarrow 0$ \citep[see][Chapter 3.9]{pete:read:1996}. For any fixed $\gamma\in \mathbb{R}$ and  arbitrarily small $A>0$, 
\begin{equation}\label{eq:invphiu}
\exp\{\gamma \Phi^{-1}(u)\}=O(\min(u,1-u)^{-A}).
\end{equation}
Since $\abs{\ln S_\tau}\le \max\{S_\tau,S_\tau^{-1}\}$, by \eqref{eq:stau} and \eqref{eq:invphiu}, we find that
\begin{equation}\label{eq:gb}
(S_\tau)^\gamma \abs{\ln S_\tau}^\beta = O(\min(u,1-u)^{-A})
\end{equation}
for any fixed $\gamma,\beta\in\mathbb{R}$.
Using \eqref{eq:lo} again, for arbitrarily small $A>0$,
\begin{align}
\frac{\partial \Phi^{-1}(u)}{\partial u}=\frac{1}{\Phi'(\Phi^{-1}(u))}&=\sqrt{2\pi}\exp\{\Phi^{-1}(u)^2/2\}\notag\\
&=\sqrt{2\pi}\exp\left[(\sqrt{-2\ln(u)}+o(1))^2/2\right]\notag\\
&=O(\min(u,1-u)^{-1-A}).\label{eq:deinvphiu}
\end{align}

Note that $\partial S_\tau/\partial \theta=a_0S_\tau$ for a constant $a_0$, and $\partial v_\tau(S_\tau,\theta)/\partial \theta$ can be expressed as the form $a_1\Phi'(-d_2)(a_2\ln S_\tau +a_3)+a_4S_\tau \Phi'(-d_1)(a_5\ln S_\tau +a_6)$ for some constants $a_i,i=1,\dots,6$. Overall, $g'_\theta(u)$ can be expressed as a function of $S_\tau$, i.e.,
\begin{equation}\label{eq:gg}
g'_\theta(u)=v'_0(\theta) - a_0[\Phi(d_1)-1]S_\tau-a_1\Phi'(-d_2)(a_2\ln S_\tau +a_3)-a_4S_\tau \Phi'(-d_1)(a_5\ln S_\tau +a_6).
\end{equation}
Using $\abs{\Phi(\cdot)}\le 1$, $\abs{\Phi'(\cdot)}\le 1/\sqrt{2\pi}$, and \eqref{eq:gb}, we find that $\abs{g'_\theta(u)}=O(\min(u,1-u)^{-A})$.
By \eqref{eq:gg},  we have
\begin{align*}
\frac{\partial g'_\theta(u)}{\partial S_\tau}=&-a_0[\Phi(d_1)-1]-a_0\Phi'(d_1)/(\sigma\sqrt{T-\tau})\\
&+a_1\Phi''(-d_2)(a_2\ln S_\tau +a_3)S_\tau^{-1}/(\sigma\sqrt{T-\tau})-a_1a_2\Phi'(-d_2)S_\tau^{-1}\\
&-a_4\Phi'(-d_1)(a_5\ln S_\tau +a_6)+a_4 \Phi''(-d_1)(a_5\ln S_\tau +a_6)/(\sigma\sqrt{T-\tau})-a_4a_5\Phi'(-d_1).
\end{align*}
By the chain rule, we have
$$\frac{\partial g'_\theta(u)}{\partial u}=\frac{\partial g'_\theta(u)}{\partial S_\tau}\frac{\partial S_\tau}{\partial u}=\frac{\partial g'_\theta(u)}{\partial S_\tau}\frac{S_\tau}{\sigma\sqrt{\tau}\Phi'(\Phi^{-1}(u))}.$$
By using \eqref{eq:gb} and \eqref{eq:deinvphiu}, and thanks to the boundedness of $\Phi''(\cdot)$, $\abs{\partial g'_\theta(u)/\partial u}=O(\min(u,1-u)^{-1-A}).$ The boundary growth condition is thus verified.

To compare the performances of Monte Carlo and RQMC based sensitivity estimators, we need to know the theoretical value of the CVaR sensitivity. Although the VaR has a closed form, it is difficult to compute the sensitivity of CVaR analytically. To get an accurate estimate as the benchmark, we use the 
RQMC-based estimate \eqref{eq:est} by replacing $\hat{v}_{\alpha,n}$ with the theoretical value of VaR $v_\alpha=g_\theta(\alpha)$ with 100 replications, each using a large sample size $n=2^{24}$. We denote this method as RQMC2, and compare it with the crude estimate \eqref{eq:est} for sample sizes $n=2^{10},\dots,2^{20}$. In our numerical experiments, we take $\alpha=0.9$. For Case 1, the benchmark is $c'_\alpha(S_0)=-0.1337$.  For Case 2, the benchmark is $c'_\alpha(r)=-3.8585$. In Figures~\ref{fig:delta1} and \ref{fig:rho1}, we compare both the mean errors and RMSEs for Monte Carlo and RQMC. We observe that the mean error and RMSE for RQMC overlap considerably with a decay rate of nearly $O(1/n)$. RQMC yields a much better error rate of convergence compared to that of Monte Carlo. 
RQMC2 is actually an RQMC quadrature $\hat I_n(f)$ for the discontinuous function $f(u)=g'_\theta(u)\bm 1\{g_\theta(u)> v_\alpha\}/(1-\alpha)$. The mean error of RQMC2 is very close to $O(1/n)$, as predicted by Proposition~\ref{prop:he}.
Comparing RQMC and RQMC2, we find that the performance of the sensitivity estimator \eqref{eq:est} is similar to that of the quadrature rule for the discontinuous function $f(u)$.  Both share an error rate of nearly $O(1/n)$. Although RQMC2 is an unbiased estimate of $c'_\alpha(\theta)$, it needs to know the true value of $v_\alpha(\theta)$ which is impossible for most cases. Without the unbiasedness, the usual estimate \eqref{eq:est} seems to be comparable to the unbiased one.

\begin{figure}[ht]
\caption{CVaR sensitivity of the put option for $\theta =S_0$ (Case 1) and $\alpha=0.9$. All errors are based on 100 replications for $n=2^{10},\dots,2^{20}$. The figure has two reference lines proportional to labeled powers of $n$. Differently from Monte Carlo (MC) and RQMC, RQMC2 uses the estimate \eqref{eq:est} by replacing $\hat{v}_{\alpha,n}$ with the true value of VaR $v_\alpha=0.859$. The benchmark is $c'_\alpha(S_0)=-0.1337$.  \label{fig:delta1}}
\centering
\includegraphics[width=0.8\hsize]{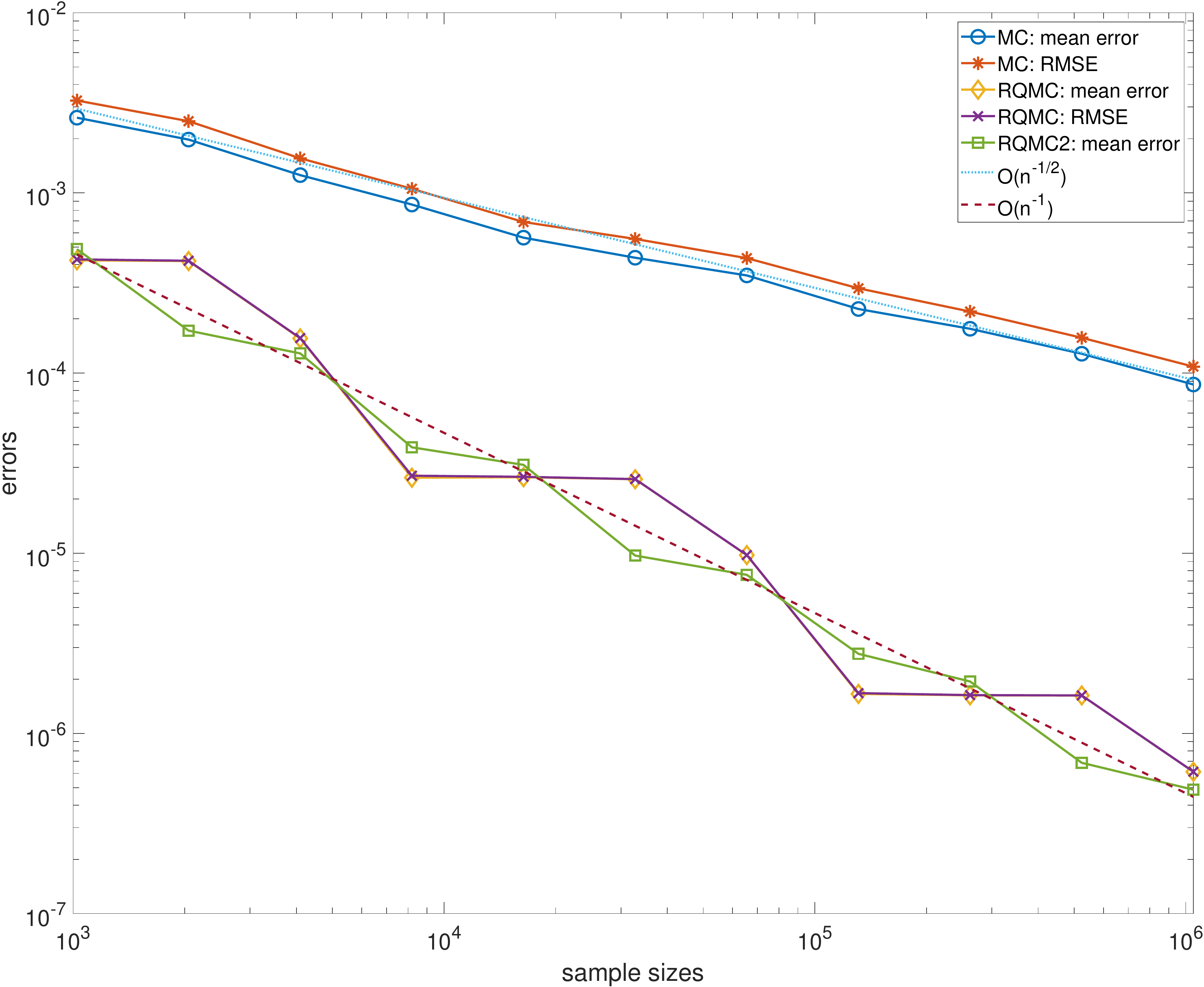}
\end{figure}

\begin{figure}[ht]
	\caption{CVaR sensitivity of the put option for $\theta =r$ (Case 2) and $\alpha=0.9$. All errors are based on 100 replications for $n=2^{10},\dots,2^{20}$. The figure has two reference lines proportional to labeled powers of $n$. Differently from Monte Carlo (MC) and RQMC, RQMC2 uses the estimate \eqref{eq:est} by replacing $\hat{v}_{\alpha,n}$ with the true value of VaR $v_\alpha=0.859$. The benchmark is $c'_\alpha(r)=-3.8585$.  \label{fig:rho1}}
	\centering
	\includegraphics[width=0.8\hsize]{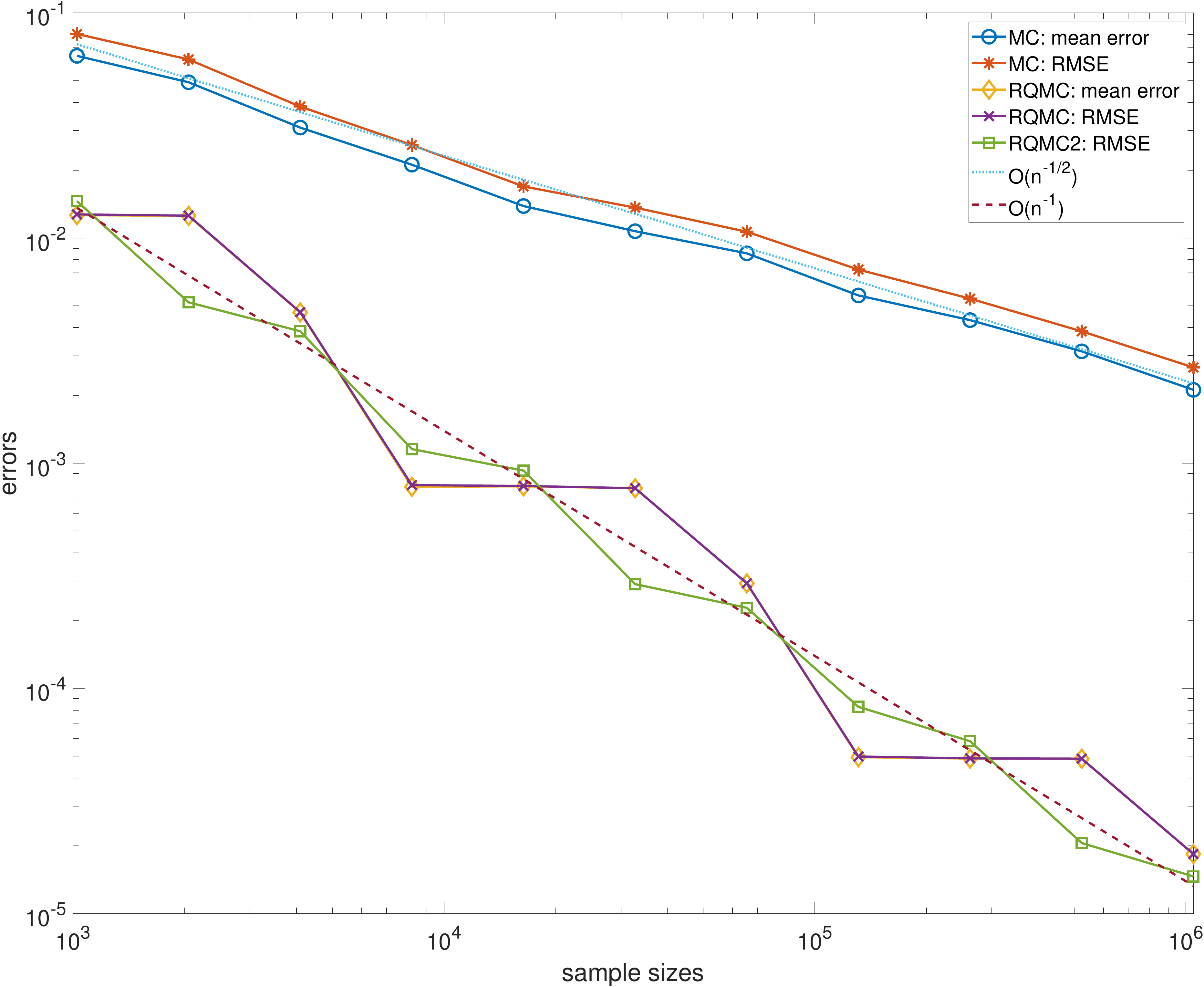}
\end{figure}

The analysis for a call option is similar. In this case, the option price at time $\tau$ is 
\begin{equation}\label{eq:tauvalue2}
v_\tau = \mbe[e^{-r(T-\tau)}(S_T-K)^+| S_\tau]= S_\tau\Phi(d_1)-Ke^{-r(T-\tau)}\Phi(d_2),
\end{equation}
and the derivative $g'_\theta(u)$ becomes
\begin{align}
g'_\theta(u)=v'_0(\theta) - \Phi(d_1)\frac{\partial S_\tau}{\partial \theta}-\frac{\partial v_\tau(S_\tau,\theta)}{\partial \theta}.\label{eq:singcall}
\end{align}
The numerical results are not reported in this paper for saving space.

\subsection{Multiple Assets}\label{sec:assets}
In this subsection, we consider a portfolio of $d$ assets $S_t^i,\ i=1,\dots,d$. The underlying asset prices follow a geometric Brownian motion. At time $\tau$,  the asset prices are 
\begin{equation}\label{eq:assets}
S_\tau^i = S_0^i\exp\{(\mu_i-\sigma_i^2/2)\tau +\sigma_iB_i(\tau)\}, i=1,\dots,d,
\end{equation}
where $\bm B(t)=(B_1(t),\dots,B_d(t))^\top$ is a $d$-dimensional Brownian motion with correlation matrix $\bm\rho=(\rho_{ij})_{i,j=1,\dots,d}$,  $\mu_i$ is the drift  under the real-world distribution for the $i$th asset, and $\sigma_i$ is the associated volatility. It is clear that $\bm B(\tau) \sim N(\bm 0,\tau\bm \rho)$. Let $\bm A$ be a decomposition of $\bm \rho$ satisfying $\bm A\bm A^{\top}=\bm\rho$. Then $\bm B(\tau)$ can be simulated via 
\begin{equation}\label{eq:A}
\bm B(\tau)= \sqrt{\tau}\bm A\bm z=\sqrt{\tau}\bm A\Phi_d^{-1}(\bm u),
\end{equation}
where   $\bm u\sim \Unif{(0,1)^d}$ and $\bm z=\Phi_d^{-1}(\bm u):=(\Phi(u_1),\dots, \Phi(u_d))^\top\sim N(\bm 0,\bm 1_d)$. By doing so, $S_\tau^i$ is a function of $\bm u$ or $\bm z$.  

Suppose that the portfolio is composed of  $J$ European options, where the $j$th option is written on the $k_j$th asset with a maturity $T_j$ and a strike $K_j$, $j=1,\dots,J$. Denote the price of the $j$th option at time $t\le T_j$ by $v^j_t=v^j_t(S_t^{k_j},\theta)$, which can be obtained via  the Black--Scholes  formula  \eqref{eq:tauvalue} or \eqref{eq:tauvalue2}. The portfolio value loss at time $\tau$ is
$$L(\theta)=g_\theta(\bm u)=\sum_{j=1}^J [v^j_0(S_0^{k_j},\theta)- v^j_\tau(S_\tau^{k_j},\theta)]=\sum_{j=1}^J g_{j,\theta}(\bm u),$$
where $g_{j,\theta}(\bm u)=v^j_0(S_0^{k_j},\theta)- v^j_\tau(S_\tau^{k_j},\theta)$ is the loss of the $j$th option. The derivative is then
$$g'_\theta(\bm u)=\sum_{j=1}^J \left[\frac{\partial v^j_0(S_0^{k_j},\theta)}{\partial\theta}- \frac{\partial v^j_\tau(S_\tau^{k_j},\theta)}{\partial \theta}\right]=\sum_{j=1}^J g'_{j,\theta}(\bm u),$$
where $g'_{j,\theta}(\bm u)$ can be obtained by \eqref{eq:singput} or \eqref{eq:singcall} depending on the type of the option. 

If $d-1$  components of $\bm u$ are fixed and the remaining one $u_j\sim\Unif{(a,b)}$, $g_\theta(\bm u)$ is a function of $u_j$ or $z_j=\Phi^{-1}(u_j)$. As a function of $z_j$, by \eqref{eq:tauvalue} and \eqref{eq:tauvalue2}, its derivative $\partial g_\theta(\bm u)/\partial z_j$ can be expressed as a linear combination of terms like $e^{c_1z_j}\Phi(c_2z_j+c_3)$ or $e^{c_4z_j}$, where $c_1,c_2,c_3,c_4$ are some constants. So the equation $\partial g_\theta(\bm u)/\partial z_j=0$ has a finite number of roots for $z_j\in\mathbb{R}$. Since $g_\theta(\bm u)$ is infinitely times differentiable with respect to $z_j$,   $g_\theta(\bm u)$ is piecewise strictly monotonic with respect to $z_j$ (or equivalently $u_j$), verifying Assumption~\ref{assum:zeroprob} (see Remark~\ref{rem:piecewise} for greater details). 

We next focus on the growth condition required in Theorem~\ref{thm:mainunboundD}. The remaining conditions in Theorem~\ref{thm:mainunboundD} can be easily verified. Since $g'_\theta(\bm u)$ is a linear combination of $g'_{j,\theta}(\bm u)$, it suffices to verify the growth condition for $g'_{j,\theta}(\bm u)$. Without loss of generality, assume that the $j$th option is a put option on the first asset, i.e., $k_j=1$. By \eqref{eq:assets} and \eqref{eq:A}, we have 
$$S_\tau^1 = S_\tau^1(\bm u) = S_0^1\exp\{(\mu_1-\sigma_1^2/2)\tau +\sigma_1\sqrt{\tau}\sum_{i=1}^da_{1i}z_i\},\ z_i=\Phi^{-1}(u_i),$$
where $a_{ij}$ are the entries of the matrix $\bm{A}$. From \eqref{eq:singput} and \eqref{eq:singcall}, we find that $g'_{j,\theta}(\bm u)$ can be expressed as a function of $S_\tau^1$, say $h(S_\tau^1(\bm u))$. For any $v\subset 1{:}d$, $\partial^v h(S_\tau^1(\bm u))$ is a linear combination of terms with the form $h^{(d')}(S_\tau^1(\bm u))\prod_{i=1}^{d'} \partial^{v_i} S_\tau^1(\bm u)$, where $d'\le d$,  $v_i\cap v_{i'}=\varnothing$ for any $i\neq i'$, and $\cup_{i=1}^{d'} v_i= v$. Note that
$$\prod_{i=1}^{d'} \partial^{v_i} S_\tau^1(\bm u)= [S_\tau^1(\bm u)]^{d'}\prod_{i\in v}\frac{\sigma_1\sqrt{\tau}a_{1i}}{\Phi'(\Phi^{-1}(u_i))}.$$
By \eqref{eq:invphiu} and \eqref{eq:deinvphiu},
$$\abs{\prod_{i=1}^{d'} \partial^{v_i} S_\tau^1(\bm u)}=O\left(\prod_{i\in v}\min(u_i,1-u_i)^{-A_i-1}\prod_{i\notin v}\min(u_i,1-u_i)^{-A_i}\right)$$
for arbitrarily small $A_i>0$. Taking the derivative of \eqref{eq:gg} $d'$ times with respect to $S_\tau$, we can see that $|h^{(d')}(S_\tau^1(\bm u))|$ is bounded by a linear combination of terms $[S_\tau^1(\bm u)]^\gamma|\ln S_\tau^1(\bm u)|^\beta$ for constants $\gamma$ and $\beta$. This is due to the fact that $\abs{\Phi^{(i)}(\cdot)}$ is bounded for any nonnegative integer $i$. Similar to \eqref{eq:gb}, for any $\gamma,\beta$, and arbitrarily small $A_i>0$, $$[S_\tau^1(\bm u)]^\gamma|\ln S_\tau^1(\bm u)|^\beta=O\left(\prod_{i=1}^d\min(u_i,1-u_i)^{-A_i}\right).$$
The growth condition for $g'_{j,\theta}(\bm u)$ is thus satisfied with arbitrarily small rates $A_i>0$.
The RQMC-based CVaR sensitivity estimate yields a mean error rate of $O(n^{-1/2-1/(4d-2)+\epsilon})$ as confirmed by Theorem~\ref{thm:mainunboundD}.

In the numerical study, we consider the following simple test portfolios:
\begin{itemize}
	\item \textbf{Portfolio A}. One call and one put options on $d=10$ independent underlying assets. Each option has a maturity of $0.25$ and a strike of $95$. Each asset has an initial value of $100$ and a volatility of $20\%$.  The real-world interest rates are $\mu_i=8\%$, and the risk-free interest rate  is $r=3\%$.
	\item \textbf{Portfolio B}. Same as Portfolio A, but with each pair of underlying assets having correlation $0.2$, i.e.,  $\rho_{ij}=0.2$ for any $i\ne j$.  
\end{itemize}
The risk horizon we choose is again $\tau=1/52$ years, and we consider $\theta=r$. For Portfolio A with independent assets, we take the matrix $\bm A=\bm 1_d$ in \eqref{eq:A}. For Portfolio B with correlated assets, we take the principal components construction \citep{glas:2004} for obtaining $\bm A$, that is $\bm A=(\sqrt{\lambda_1}\bm v_1,\dots,\sqrt{\lambda_d}\bm v_d)$ where $\lambda_1\ge \dots\ge \lambda_d$ are the eigenvalues of the correlation matrix $\bm \rho$ and $\bm{v}_1,\dots,\bm v_d$ are the corresponding eigenvectors of unit length. To get an accurate estimate as the benchmark, we run the RQMC method with 100 replications, each using a large sample size $n=2^{22}$. Figures~\ref{fig:rho10} and \ref{fig:rho10depend} show the convergence results of Monte Carlo and RQMC for the two portfolios, respectively. It is clear that RQMC performs better than Monte Carlo. The mean error rate of RQMC diminishes compared to the case of a single option. As predicted by the theoretical rate $O(n^{-1/2-1/(4d-2)+\epsilon})$, RQMC suffers from the curse of dimensionality although the rate is asymptotically better than the Monte Carlo rate $O(n^{-1/2})$. 

Sensitivity estimation of CVaR is more challenging for RQMC due to both discontinuity and high dimensionality. To overcome the impact of high dimensionality, some dimension reduction strategies are proposed in the literature \citep{wang2013pricing,weng:2016}. It is widely believed that  RQMC can be very effective if the effective dimensions of the function is low \citep{caflisch1997valuation}.  The decomposition of  the correlation matrix $\bm \rho=\bm A\bm A^\top$ is not unique, and it has an impact on the effective dimensions of the two important functions $g'_\theta(\bm u)\bm 1\{g_\theta(\bm u)> v_\alpha\}$ and $\bm 1\{g_\theta(\bm u)\le v_\alpha\}$. It therefore leaves room for choosing a proper matrix $\bm A$ in \eqref{eq:A} to reduce the effective dimensions. On the other hand, to overcome the impact of discontinuity, one may resort to some smoothing methods, such as conditioning \citep{zhang2019quasi}. RQMC enjoys a faster rate of convergence if the integrand is sufficiently smooth. \cite{he_error_2019} showed that RQMC together with conditioning achieves a mean error of $O(n^{-1+\epsilon})$   for option pricing problems. We leave these strategies of improving RQMC efficiency for future research. 

\begin{figure}[ht]
	\caption{CVaR sensitivity of Portfolio A (independent assets) for $\theta =r$ and $\alpha=0.9$. All errors are based on 100 replications for $n=2^{10},\dots,2^{20}$. The figure has a reference line $O(n^{-1/2})$. The benchmark is $c'_\alpha(r)=8.0814$.  \label{fig:rho10}}
	\centering
	\includegraphics[width=0.8\hsize]{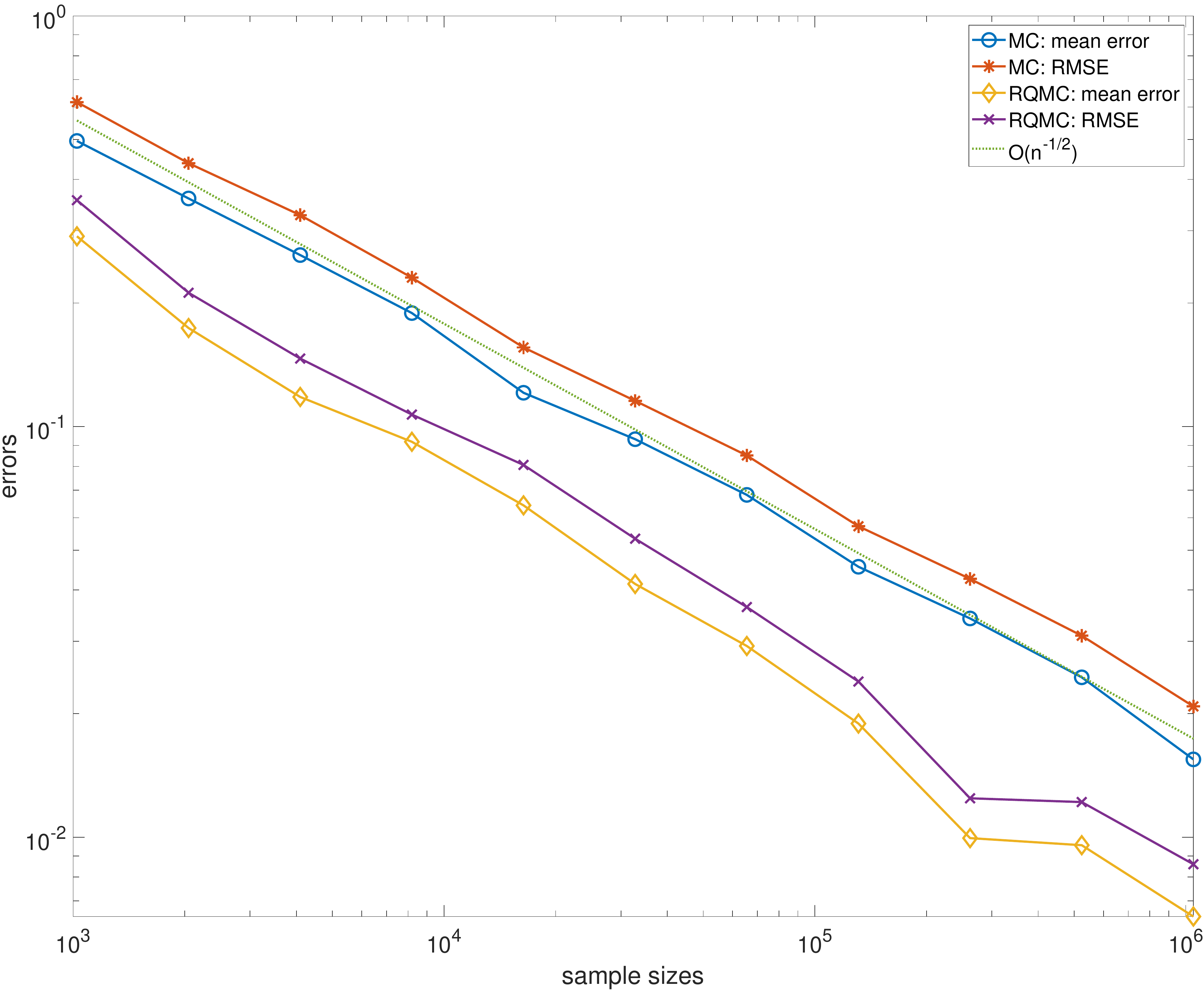}
\end{figure}

\begin{figure}[ht]
	\caption{CVaR sensitivity of Portfolio B (correlated assets) for $\theta =r$ and $\alpha=0.9$. All errors are based on 100 replications for $n=2^{10},\dots,2^{20}$. The figure has a reference line $O(n^{-1/2})$. The benchmark is $c'_\alpha(r)=15.1564$.  \label{fig:rho10depend}}
	\centering
	\includegraphics[width=0.8\hsize]{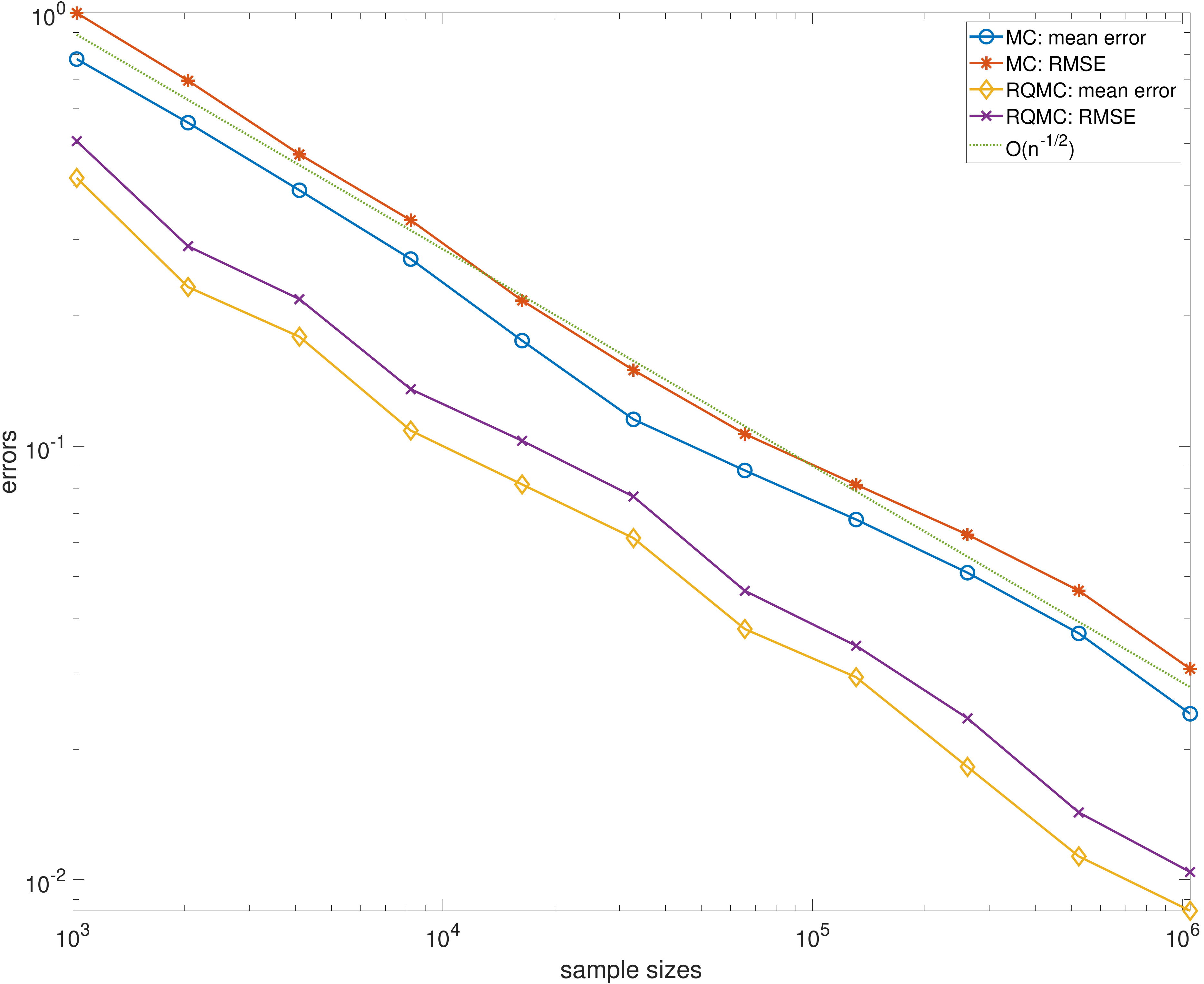}
\end{figure}
\subsection{Delta-Gamma Approximation}
Consider a portfolio of many assets (such as stocks, options) that depend on $d$ risk factors. Let $\Delta S$ denote the changes in the risk factors in a given time period $\tau$. \cite{hong:liu:2009} studied a quadratic model
\begin{equation}\label{eq:quad}
L = a_0+\bm \alpha^\top\Delta S+\Delta S^\top \bm A \Delta S,
\end{equation}
where $a_0\in\mathbb{R}$, $\bm \alpha=(\alpha_1,\dots,\alpha_d)^\top$ and  $\bm A\in \mathbb{R}^{d\times d}$ are known. Assume that $\bm A$ is positive definite, and $\Delta S\sim N(\bm \mu,\bm \Sigma)$ with mean $\bm \mu=(\mu_1,\dots,\mu_d)^\top$ and covariance $\bm \Sigma$ (also  positive definite). The simple model \eqref{eq:quad} is actually the delta-gamma approximation of the loss studied in Section~\ref{sec:assets}. \cite{glasserman_variance_2000} used the approximation to guide the selection of effective variance reduction techniques in estimating VaR. 

The mean and covariance are estimated from historical data. We are interested in estimating CVaR sensitive  to the parameter $\theta =\mu_k$ for some $k\le d$. Let $\bm \Sigma=\bm C\bm C^\top$ be a decomposition of the covariance.
We may write  $$\Delta S=\bm\mu + \bm C\bm z=\bm\mu +\bm C\Phi^{-1}_d(\bm u),$$ where $\bm u\sim \Unif{(0,1)^d}$ and $\bm z=\Phi^{-1}_d(\bm u)\sim N(\bm 0,\bm 1_d)$. It then gives
\begin{equation*}\label{eq:quad2}
L(\theta) = a_0+\bm \alpha^\top\bm \mu+\bm\mu^\top \bm A\bm \mu +(\bm \alpha^\top \bm C+2\bm \mu^\top \bm A\bm C)\bm z+\bm z^\top \bm C^\top\bm A \bm C\bm z.
\end{equation*}
If $d-1$  components of $\bm z$ are fixed, $L$ is a quadratic function of the remaining component, which is piecewise strictly monotonic. Assumption~\ref{assum:zeroprob} is therefore verified.
The derivative is then
\begin{equation*}
L'(\theta)=g'_\theta(\bm u)=\alpha_k+2\bm A_{k\cdot}^\top \bm \mu +2 (\bm A\bm C)_{k\cdot}\Phi^{-1}_d(\bm u),
\end{equation*}
where $\bm A_{k\cdot}$ denotes the $k$th row of the matrix $\bm A$ and similarly for $(\bm A\bm C)_{k\cdot}$. It then follows by \eqref{eq:invphiu} and \eqref{eq:deinvphiu} that the growth condition for $g'_{\theta}(\bm u)$ is satisfied with arbitrarily small rates. 
The RQMC-based CVaR sensitivity estimate yields a mean error rate of $O(n^{-1/2-1/(4d-2)+\epsilon})$ as confirmed by Theorem~\ref{thm:mainunboundD}. The quadratic model \eqref{eq:quad} is much simpler than the model in Section~\ref{sec:assets}. We  do not report the numerical results here.

\section{Conclusion}\label{sec:concl}
In this paper we found convergence rates of RQMC for CVaR sensitivity estimation. 
The theoretical results show that RQMC yields an asymptotically faster error rate than Monte Carlo, but the rate deteriorates  as the dimension $d$ increases. It is important to  note that the results we proved are the worst-case error rates.
A good performance  can be expected from RQMC in high dimensions if  dimension reduction methods and smoothing methods are well utilized. We hope that the established theoretical results could be helpful for guiding a good way to improve plain RQMC.

\section*{Acknowledgments}

This work was supported by the National Science Foundation of China (No. 12071154) and the Fundamental Research Funds for the Central Universities (No. 2019MS106).
\section{Appendix}

\begin{lemma}\label{lem:append}
Let $L_i = g_\theta(\tilde{\bm u}_i)$, where $\{\tilde{\bm u}_1,\dots,\tilde{\bm u}_n\}$ is a scrambled $(t,m,d)$-net in base $b\ge 2$ with $n=b^m$.	If Assumption~\ref{assum:crv} is satisfied, there are at most $b^t$ of the observations $L_1,\dots,L_n$ with equal value w.p.1.
\end{lemma}

\begin{proof}
Suppose that $\tilde{\bm u}_1,\dots,\tilde{\bm u}_n$ are obtained by scrambling  $\bm u_1,\dots,\bm u_n$ as described in \eqref{eq:scrambling}.
Assume that $\bm u_i\neq \bm u_{i'}$ with unequal components on the $j$th coordinate, i.e., $u_i^j\neq u_{i'}^j$. We write  $u_i^j = \sum_{k=1}^{\infty} a_{ijk}b^{-k}$ and $u_{i'}^j = \sum_{k=1}^{\infty} a_{i'jk}b^{-k}$ in the $b$-adic expansion, where $a_{ijk},a_{i'jk}\in \{0,1,\dots,b-1\}$. Let $s=\min\{k\ge 1: a_{ijk}\neq  a_{i'jk}\}$. For all $k<s$, $a_{ijk}=  a_{i'jk}$, but $a_{ijs}\neq  a_{i'js}$. Denote $\tilde{a}_{ijk}$ and $\tilde{a}_{i'jk}$ as  random permutations of $a_{ijk}$ and $a_{i'jk}$, respectively. According to the scrambling procedure \eqref{eq:scrambling},  for all $k\ge s+1$, $\tilde{a}_{ijk}$ and  $\tilde{a}_{i'jk}$ are independent because different permutations are applied. Recall that all  permutations  are independent.  This implies that conditional on $\tilde{a}_{ijk},\tilde{a}_{i'jk}$ for all $k\le s$, $\tilde{u}_i^j\sim \Unif{(\sum_{k=1}^{s} \tilde{a}_{ijk}b^{-k},1)}$ and $\tilde{u}_{i'}^j\sim \Unif{(\sum_{k=1}^{s} \tilde{a}_{i'jk}b^{-k},1)}$ independently. This uniformity property can be proved as in the proof of \citep[Proposition 2]{owen:1995}, which showed  each $\tilde{\bm u}_i\sim \Unif{[0,1)^d}$.   By Assumption~\ref{assum:crv}, conditional on all $\tilde u_i^\ell,\tilde u_{i'}^\ell$ with $\ell\ne j$ and $\tilde{a}_{ijk},\tilde{a}_{i'jk}$ for $k\le s$ (denote these information as $\mathcal{F}$), $L_i$ and $L_{i'}$ are independent continuous random variables, implying  $\mbp(L_i=L_{i'}|\mathcal{F})=0$. By the law of total expectation, we have
$$\mbp(L_i=L_{i'})=\mbe[\mbp(L_i=L_{i'}|\mathcal{F})]=0.$$

By the definition of $(t,m,d)$-net, there are at most $b^t$ points of $\bm u_1,\dots,\bm u_n$ with equal value. For any two distinct points, the associated random observations of $L$ are different w.p.1. Consequently, there are at most $b^t$ of the observations $L_1,\dots,L_n$ with equal value w.p.1.
\end{proof}

\begin{remark}\label{rem:sobol}
The constant $b^t$ may be conservative for some cases of $t> 0$. For example, a Sobol' sequence is a $(t,d)$-sequence in base $b=2$. The value of $t$ depends on $d$ and is larger than 0 for moderately large $d$; see \cite{dick:2008} for detailed discussion. Direction numbers for generating Sobol' sequences that satisfy the so-called Property A in up to 1111 dimensions have  been given in \cite{joe_constructing_2008}. Property A was introduced by \cite{sobol1976uniformly}. If the unit cube $[0,1)^d$ is divided by the planes $x_j = 1/2$ into $2^d$ equally-sized subcubes, then a sequence of
points in $[0, 1)^d$ satisfies Property A if, after dividing the sequence into
consecutive blocks of $2^d$ points, each one of the points in any block belongs to a
different subcube. This property guarantees that all the points are distinct. For this case, from the proof of Lemma~\ref{lem:append}, the constant $b^t$ can be replaced by $1$.
\end{remark}


\begin{thebibliography}{46}
	\providecommand{\natexlab}[1]{#1}
	\providecommand{\url}[1]{\texttt{#1}}
	\expandafter\ifx\csname urlstyle\endcsname\relax
	\providecommand{\doi}[1]{doi: #1}\else
	\providecommand{\doi}{doi: \begingroup \urlstyle{rm}\Url}\fi
	
	\bibitem[Ambrosio et~al.(2008)Ambrosio, Colesanti, and Villa]{Ambrosio2008}
	L.~Ambrosio, A.~Colesanti, and E.~Villa.
	\newblock Outer {M}inkowski content for some classes of closed sets.
	\newblock \emph{Mathematische Annalen}, 342\penalty0 (4):\penalty0 727--748,
	2008.
	
	\bibitem[Asimit et~al.(2019)Asimit, Peng, Wang, and Yu]{asim:2019}
	V.~Asimit, L.~Peng, R.~Wang, and A.~Yu.
	\newblock An efficient approach to quantile capital allocation and sensitivity
	analysis.
	\newblock \emph{Mathematical Finance}, pages 1--26, 2019.
	
	\bibitem[Avramidis and Wilson(1998)]{avramidis_correlation-induction_1998}
	A.~N. Avramidis and J.~R. Wilson.
	\newblock Correlation-induction techniques for estimating quantiles in
	simulation experiments.
	\newblock \emph{Operations Research}, 46\penalty0 (4):\penalty0 574--591, 1998.
	
	\bibitem[Broadie and Glasserman(1996)]{broa:glass:1996}
	M.~Broadie and P.~Glasserman.
	\newblock Estimating security price derivatives using simulation.
	\newblock \emph{Management science}, 42\penalty0 (2):\penalty0 269--285, 1996.
	
	\bibitem[Broadie et~al.(2011)Broadie, Du, and Moallemi]{broadie_efficient_2011}
	M.~Broadie, Y.~Du, and C.~C. Moallemi.
	\newblock Efficient risk estimation via nested sequential simulation.
	\newblock \emph{Management Science}, 57\penalty0 (6):\penalty0 1172--1194,
	2011.
	
	\bibitem[Caflisch et~al.(1997)Caflisch, Morokoff, and
	Owen]{caflisch1997valuation}
	R.~E. Caflisch, W.~J. Morokoff, and A.~B. Owen.
	\newblock Valuation of mortgage backed securities using {B}rownian bridges to
	reduce effective dimension.
	\newblock \emph{Journal of Computational Finance}, 1\penalty0 (1):\penalty0
	27--46, 1997.
	
	\bibitem[Cui et~al.(2019)Cui, Fu, Hu, Liu, Peng, and Zhu]{cui2019variance}
	Z.~Cui, M.~C. Fu, J.~Hu, Y.~Liu, Y.~Peng, and L.~Zhu.
	\newblock On the variance of single-run unbiased stochastic derivative
	estimators.
	\newblock \emph{INFORMS Journal on Computing}, 32\penalty0 (2):\penalty0
	390--407, 2019.
	
	\bibitem[Dick and Niederreiter(2008)]{dick:2008}
	J.~Dick and H.~Niederreiter.
	\newblock On the exact t-value of {{Niederreiter}} and {{Sobol}}' sequences.
	\newblock \emph{Journal of Complexity}, 24\penalty0 (5-6):\penalty0 572--581,
	2008.
	
	\bibitem[Dick and Pillichshammer(2010)]{dick:2010}
	J.~Dick and F.~Pillichshammer.
	\newblock \emph{Digital Nets and Sequences: Discrepancy Theory and Quasi--Monte
		Carlo Integration}.
	\newblock Cambridge University Press, 2010.
	
	\bibitem[Dong and Nakayama(2017)]{dong_quantile_2017}
	H.~Dong and M.~K. Nakayama.
	\newblock Quantile estimation with latin hypercube sampling.
	\newblock \emph{Operations Research}, 65\penalty0 (6):\penalty0 1678--1695,
	2017.
	
	\bibitem[Fox and Glynn(1989)]{fox:glynn:1989}
	B.~L. Fox and P.~W. Glynn.
	\newblock Replication schemes for limiting expectations.
	\newblock \emph{Probability in the Engineering and Informational Sciences},
	3:\penalty0 299--318, 1989.
	
	\bibitem[Fu et~al.(2009)Fu, Hong, and Hu]{fu2009conditional}
	M.~C. Fu, L.~J. Hong, and J.-Q. Hu.
	\newblock Conditional {M}onte {C}arlo estimation of quantile sensitivities.
	\newblock \emph{Management Science}, 55\penalty0 (12):\penalty0 2019--2027,
	2009.
	
	\bibitem[Glasserman(2004)]{glas:2004}
	P.~Glasserman.
	\newblock \emph{Monte {C}arlo Methods in Financial Engineering}.
	\newblock Springer, 2004.
	
	\bibitem[Glasserman et~al.(2000)Glasserman, Heidelberger, and
	Shahabuddin]{glasserman_variance_2000}
	P.~Glasserman, P.~Heidelberger, and P.~Shahabuddin.
	\newblock Variance reduction techniques for estimating value-at-risk.
	\newblock \emph{Management Science}, 46\penalty0 (10):\penalty0 1349--1364,
	2000.
	
	\bibitem[Gordy and Juneja(2010)]{gordy_nested_2010}
	M.~B. Gordy and S.~Juneja.
	\newblock Nested simulation in portfolio risk measurement.
	\newblock \emph{Management Science}, 56\penalty0 (10):\penalty0 1833--1848,
	2010.
	
	\bibitem[He(2018)]{he:2018}
	Z.~He.
	\newblock Quasi-{{Monte Carlo}} for discontinuous integrands with singularities
	along the boundary of the unit cube.
	\newblock \emph{Mathematics of Computation}, 87\penalty0 (314):\penalty0
	2857--2870, 2018.
	
	\bibitem[He(2019)]{he_error_2019}
	Z.~He.
	\newblock On the error rate of conditional quasi--{{Monte Carlo}} for
	discontinuous functions.
	\newblock \emph{SIAM Journal on Numerical Analysis}, 57\penalty0 (2):\penalty0
	854--874, 2019.
	
	\bibitem[He and Wang(2015)]{he:wang:2015}
	Z.~He and X.~Wang.
	\newblock On the convergence rate of randomized quasi--{M}onte {C}arlo for
	discontinuous functions.
	\newblock \emph{SIAM Journal on Numerical Analysis}, 53\penalty0 (5):\penalty0
	2488--2503, 2015.
	
	\bibitem[He and Wang(2020)]{he_convergence_2017}
	Z.~He and X.~Wang.
	\newblock Convergence analysis of quasi-{{Monte Carlo}} sampling for quantile
	and expected shortfall.
	\newblock \emph{Mathematics of Computation}, 2020.
	\newblock Appeared online.
	
	\bibitem[Hong(2009)]{hong2009estimating}
	L.~J. Hong.
	\newblock Estimating quantile sensitivities.
	\newblock \emph{Operations research}, 57\penalty0 (1):\penalty0 118--130, 2009.
	
	\bibitem[Hong and Liu(2009)]{hong:liu:2009}
	L.~J. Hong and G.~Liu.
	\newblock Simulating sensitivities of conditional value at risk.
	\newblock \emph{Management Science}, 55\penalty0 (2):\penalty0 281--293, 2009.
	
	\bibitem[Hong et~al.(2014)Hong, Juneja, and Luo]{hong2014estimating}
	L.~J. Hong, S.~Juneja, and J.~Luo.
	\newblock Estimating sensitivities of portfolio credit risk using {M}onte
	{C}arlo.
	\newblock \emph{INFORMS Journal on Computing}, 26\penalty0 (4):\penalty0
	848--865, 2014.
	
	\bibitem[Hull(2015)]{hull:2015}
	J.~C. Hull.
	\newblock \emph{Options, Futures, and Other Derivatives}.
	\newblock Pearson, 2015.
	
	\bibitem[Jiang and Fu(2015)]{jiang2015estimating}
	G.~Jiang and M.~C. Fu.
	\newblock Technical note--{O}n estimating quantile sensitivities via
	infinitesimal perturbation analysis.
	\newblock \emph{Operations Research}, 63\penalty0 (2):\penalty0 435--441, 2015.
	
	\bibitem[Joe and Kuo(2008)]{joe_constructing_2008}
	S.~Joe and F.~Y. Kuo.
	\newblock Constructing {{Sobol}} sequences with better two-dimensional
	projections.
	\newblock \emph{SIAM Journal on Scientific Computing}, 30\penalty0
	(5):\penalty0 2635--2654, 2008.
	
	\bibitem[Joy et~al.(1996)Joy, Boyle, and Tan]{joy1996quasi}
	C.~Joy, P.~P. Boyle, and K.~S. Tan.
	\newblock Quasi-{M}onte {C}arlo methods in numerical finance.
	\newblock \emph{Management Science}, 42\penalty0 (6):\penalty0 926--938, 1996.
	
	\bibitem[Kaplan et~al.(2019)Kaplan, Li, Nakayama, and Tuffin]{kaplan:2019}
	Z.~Kaplan, Y.~Li, M.~Nakayama, and B.~Tuffin.
	\newblock Randomized quasi-{M}onte {C}arlo for quantile estimation.
	\newblock In \emph{Proceedings of the 2019 Winter Simulation Conference}, pages
	1--14, 2019.
	
	\bibitem[L'Ecuyer and Lemieux(2005)]{lecu:lemi:2005}
	P.~L'Ecuyer and C.~Lemieux.
	\newblock Recent advances in randomized quasi-{M}onte {C}arlo methods.
	\newblock In M.~Dror, P.~L'Ecuyer, and F.~Szidarovszky, editors, \emph{Modeling
		Uncertainty: An Examination of Stochastic Theory, Methods, and Applications},
	pages 419--474. Kluwer Academic Publishers, 2005.
	
	\bibitem[Matou{\v{s}}ek(1998)]{mato:1998}
	J.~Matou{\v{s}}ek.
	\newblock On the ${L}_2$-discrepancy for anchored boxes.
	\newblock \emph{Journal of Complexity}, 14\penalty0 (4):\penalty0 527--556,
	1998.
	
	\bibitem[Niederreiter(1992)]{nied:1992}
	H.~Niederreiter.
	\newblock \emph{Random {Number} {Generation} and {Quasi}-{Monte} {Carlo}
		{Methods}}.
	\newblock SIAM, Philadelphia, 1992.
	
	\bibitem[Owen(1995)]{owen:1995}
	A.~B. Owen.
	\newblock Randomly permuted $(t, m, s)$-nets and $(t, s)$-sequences.
	\newblock In H.~Niederreiter and P.~J.-S. Shiue, editors, \emph{Monte Carlo and
		Quasi-Monte Carlo Methods in Scientific Computing}, pages 299--317. Springer,
	1995.
	
	\bibitem[Owen(1997{\natexlab{a}})]{owen:1997b}
	A.~B. Owen.
	\newblock Monte {C}arlo variance of scrambled net quadrature.
	\newblock \emph{SIAM Journal Numerical Analysis}, 34\penalty0 (5):\penalty0
	1884--1910, 1997{\natexlab{a}}.
	
	\bibitem[Owen(1997{\natexlab{b}})]{owen_scrambled_1997}
	A.~B. Owen.
	\newblock Scrambled net variance for integrals of smooth functions.
	\newblock \emph{The Annals of Statistics}, 25\penalty0 (4):\penalty0
	1541--1562, 1997{\natexlab{b}}.
	
	\bibitem[Owen(2005)]{owen:2005}
	A.~B. Owen.
	\newblock Multidimensional variation for quasi-{M}onte {C}arlo.
	\newblock In J.~Fan and G.~Li, editors, \emph{International Conference on
		Statistics in honour of Professor Kai-Tai Fang's 65th birthday}, pages
	49--74, 2005.
	
	\bibitem[Owen(2006)]{owen_halton_2006}
	A.~B. Owen.
	\newblock Halton sequences avoid the origin.
	\newblock \emph{SIAM Review}, 48\penalty0 (3):\penalty0 487--503, 2006.
	
	\bibitem[Owen(2008)]{owen_local_2008}
	A.~B. Owen.
	\newblock Local antithetic sampling with scrambled nets.
	\newblock \emph{The Annals of Statistics}, 36\penalty0 (5):\penalty0
	2319--2343, 2008.
	
	\bibitem[Owen and Rudolf(2020)]{owen:2020}
	A.~B. Owen and D.~Rudolf.
	\newblock A strong law of large numbers for scrambled net integration.
	\newblock \emph{arXiv preprint arXiv:2002.07859}, 2020.
	
	\bibitem[Patel and Read(1996)]{pete:read:1996}
	J.~K. Patel and C.~B. Read.
	\newblock \emph{Handbook of the Normal Distribution}, volume 150.
	\newblock Marcel Dekker, New York, 1996.
	
	\bibitem[Peng et~al.(2018)Peng, Fu, Hu, and Heidergott]{peng2018new}
	Y.~Peng, M.~C. Fu, J.-Q. Hu, and B.~Heidergott.
	\newblock A new unbiased stochastic derivative estimator for discontinuous
	sample performances with structural parameters.
	\newblock \emph{Operations Research}, 66\penalty0 (2):\penalty0 487--499, 2018.
	
	\bibitem[Scaillet(2004)]{scail:2004}
	O.~Scaillet.
	\newblock Nonparametric estimation and sensitivity analysis of expected
	shortfall.
	\newblock \emph{Mathematical Finance}, 14\penalty0 (1):\penalty0 115--129,
	2004.
	
	\bibitem[Serfling(1980)]{serf:1980}
	R.~J. Serfling.
	\newblock \emph{Approximation theorems of mathematical statistics}.
	\newblock Wiley, New York, 1980.
	
	\bibitem[Sobol(1976)]{sobol1976uniformly}
	I.~M. Sobol.
	\newblock Uniformly distributed sequences with an additional uniform property.
	\newblock \emph{USSR Computational Mathematics and Mathematical Physics},
	16\penalty0 (5):\penalty0 236--242, 1976.
	
	\bibitem[Wang and Tan(2013)]{wang2013pricing}
	X.~Wang and K.~S. Tan.
	\newblock Pricing and hedging with discontinuous functions: Quasi--monte carlo
	methods and dimension reduction.
	\newblock \emph{Management Science}, 59\penalty0 (2):\penalty0 376--389, 2013.
	
	\bibitem[Weng et~al.(2016)Weng, Wang, and He]{weng:2016}
	C.~Weng, X.~Wang, and Z.~He.
	\newblock An auto-realignment method in quasi-{{Monte Carlo}} for pricing
	financial derivatives with jump structures.
	\newblock \emph{European Journal of Operational Research}, 254\penalty0
	(1):\penalty0 304--311, 2016.
	
	\bibitem[Xie et~al.(2019)Xie, He, and Wang]{xie2019importance}
	F.~Xie, Z.~He, and X.~Wang.
	\newblock An importance sampling-based smoothing approach for quasi-{{M}}onte
	{{C}}arlo simulation of discrete barrier options.
	\newblock \emph{European Journal of Operational Research}, 274\penalty0
	(2):\penalty0 759--772, 2019.
	
	\bibitem[Zhang and Wang(2019)]{zhang2019quasi}
	C.~Zhang and X.~Wang.
	\newblock Quasi-{M}onte {C}arlo-based conditional pathwise method for option
	greeks.
	\newblock \emph{Quantitative Finance}, pages 1--19, 2019.
	
\end{thebibliography}
\end{document}